\documentclass[12pt]{amsart}
\usepackage{a4,amsmath,amssymb,epic}
\usepackage{graphicx}
\usepackage{multicol}
\usepackage{rotating}
\usepackage[margin=3cm]{geometry}
\usepackage{amssymb,amsthm}
\usepackage{enumitem}
  \usepackage[all]{xy}
\usepackage{setspace}

\allowdisplaybreaks
\numberwithin{equation}{subsection}
\theoremstyle{plain}
\newtheorem*{lem}{Lemma}
\newtheorem*{prop}{Proposition}
\newtheorem*{thm}{Theorem}

\newtheorem*{murthm}{A semistandard basis theorem}
\newtheorem*{cor}{Corollary}

\newtheorem*{lem*}{Lemma}
\newtheorem*{Acknowledgements*}{Acknowledgements}
\newtheorem*{prop*}{Proposition}
\newtheorem*{thm*}{Theorem}
\newtheorem*{cor*}{Corollary}
\newtheorem*{conj*}{Conjecture}
\theoremstyle{remark}

\newtheorem*{rmk}{Remark}
\newtheorem*{Remark on the proof}{Remark on the proof}

\newtheorem{remk}{Remark}
\newtheorem*{defn}{Definition}
\newtheorem*{eg}{Example}

\newcommand{\Ext}{\operatorname{Ext}}

\newcommand{\Hom}{\operatorname{Hom}}
\newcommand{\Head}{\operatorname{head}}
\newcommand{\Stab}{\operatorname{Stab}}
\newcommand{\Span}{\operatorname{Span}}

\newcommand{\minl}{\operatorname{min}}
\newcommand{\End}{\operatorname{End}}
\newcommand{\Ind}{\operatorname{Ind}}

\newcommand{\Res}{\operatorname{Res}}

\newcommand{\GL}{\mathrm{GL}}

\newcommand{\T}{\mathrm{T}}
\renewcommand{\S}{\mathrm{S}}
\newcommand{\V}{\mathrm{V}}
\newcommand{\U}{\mathrm{U}}
\newcommand{\F}{\mathrm{F}}

\newcommand{\G}{\mathrm{G}}

\newcommand{\Std}{\rm{Std}}

\newcommand{\id}{\mathrm{id}}

\begin{document}

\title[Bases of quasi-hereditary covers of diagram algebras]{Bases of quasi-hereditary covers \\ of diagram algebras}
\author{C.~Bowman}
\address{Bowman@math.jussieu.fr}
\date{10th July 2012}

\begin{abstract} 
 We extend the the combinatorics of tableaux to the study of Brauer, walled Brauer, and partition algebras.
In particular, we provide uniform constructions of Murphy bases and `Specht' filtrations of permutation modules.
This allows us to give a uniform construction of semistandard bases of their quasi-hereditary covers.
  \end{abstract}
\maketitle
\section*{Introduction}
The Schur algebra was defined by  Green, \cite{Green}, as the setting in which to study the 
mutually centralising actions of the general linear and symmetric group on tensor space.  
Hemmer and Nakano, \cite{hemnak}, showed that when we restrict to the subcategories of cell-filtered modules $\Ext^1$-information can be passed through the Schur functor (for $p\neq 2,3$).

This was generalised in \cite{rouq}, where Rouquier introduced covers of finite dimensional algebras by highest weight categories and considered different levels of `faithfulness' in  cohomology. The simplest situation is that of a double centraliser theorem.    
 
 The Dipper--James--Mathas approach to the construction of a quasi-hereditary cover is to give a Murphy basis of `permutation modules' and to `lift' this to a cellular basis of the endomorphism algebra.  
   This is the approach used to construct the cyclotomic Schur algebras as covers of the Hecke algebras of type $G(m,1,r)$. 
    These algebras have   no underlying geometry and yet have all the symptoms of Lie theory:
a highest weight theory;
 Jantzen sum formulae; 
PBW bases;
  and are conjectured to be Koszul in the abelian defect case.  
This Lie theory apparently arises from the combinatorics of the underlying Hecke algebra.

 More generally, one can ask when and how can we directly construct a quasi-hereditary cover of a cellular algebra, and how much combinatorial Lie theory is baked-in to this construction. 
 In this paper we give a uniform approach (in the spirit of Dipper, James, and Mathas) to the construction of the quasi-hereditary covers of the Brauer, walled Brauer, and partition algebras.

Over fields of characteristic zero many classical diagram algebras are quasi-hereditary (i.e.\! our Schur functor induces a Morita equivalence).  This is no longer true for quantised diagram algebras at roots of unity, or classical diagram algebras (such as the Brauer, walled Brauer, and partition algebras) over fields of modular characteristic.
The principle aim of this paper is to construct a characteristic-free setting in which to study these diagram algebras through the language of Lie theory.

Since the first version of this paper was made available, the paper \cite{hko} has now appeared online, in which similar results  are obtained in the case of the Brauer algebra.

\section{Cellularly stratified diagram algebras}
We fix an algebraically close field, $K$, of characteristic $p\geq 0$.  We recall the known results concerning cellularly stratified algebras that we shall require for the remainder of the paper, see \cite{HHKP} for more details.

\subsection{Cellular algebras and iterated inflations}We recall the original definition of a cellular algebra given by Graham and Lehrer in \cite{GL}.

\begin{defn}\label{CELLDEF} 
An associative $K$-algebra $A$ is called a cellular algebra with cell datum $(\Lambda;T;C; i)$ if the following conditions are satisfied:

(C1) The finite set $\Lambda$ is partially ordered.  Associated with each $\lambda \in \Lambda$ there is a finite set $M(\lambda)$.  The algebra $A$ has $K$-basis $C^\lambda _{S,T}$ where $(S,T)$ runs through all elements of $M(\lambda)$ for all $\lambda \in \Lambda$.

(C2) The map $i$ is a $K$-linear anti-automorphism of $A$ with $i^2 = \id$ which sends
each $C^\lambda_{S,T}$ to $C^\lambda_{T,S}$.

(C3) For each $\lambda \in \Lambda$ and $S,T \in T(\lambda)$ and each $a \in A$ the product $aC^\lambda_{S,T}$ can be written as $(\sum_{U \in M(\lambda)}r_a(U,S)C^\lambda_{U,T}) + r'$ where $r'$ is a linear combination of basis elements with upper index strictly less than $\lambda$, and where the $r_a(U, S) \in K$ do not depend on $T$.
\end{defn}

In \cite{KX} it was shown  that every cellular algebra can be constructed as an iterated inflation of smaller cellular algebras. Let $A$ be an algebra which can be realised as an iterated inflation of cellular algebras $B_l$ along vector spaces $V_l$ for $l = 1, \ldots, n$.  As a vector space
$$A = \oplus^n_{l=0} V_l \otimes V_l \otimes B_l.$$   

An element of $A$ is called a \emph{diagram}, an element of $V_l$ is called a \emph{dangle}, and an element of $B_l$ is called a \emph{configuration of through lines}.  Any diagram is formed from $u$ a top dangle, $v$ a bottom dangle, and $b$ a configuration of through lines.
  There is a chain of two-sided ideals $A=J_0 \supseteq J_1 \supseteq \ldots \supseteq J_n=\{0\}$, which can be refined to a cell chain, and each subquotient $J_l / J_{l -1}$ equals $B_l \otimes V_l \otimes V_l$ as an algebra without unit. The anti-involution $i$ of A 
is defined through the anti-involutions $j_l$ of the cellular algebras $B_l$ as follows: $i(u\otimes v\otimes b)=  v \otimes u \otimes j_l(b)$
for any $b \in B_l$ and $u,v \in V_l$.  


Suppose that the \emph{input algebra} $B_l$ has cell modules $\{\mathcal{S}(\lambda) : \lambda\in\Lambda_{B_l}\}$, then $A$ has cell modules $\{\Delta(\lambda,l)=V_l \otimes \mathcal{S}(\lambda) : (l,\lambda) \in\Lambda_A\}$.

\subsection{Cellularly stratified algebras}\label{stratified}
The following definition first appears in \cite{HHKP}.  

\begin{defn}\label{cellDEF}
A finite dimensional associative algebra $A$ is called \emph{cellularly stratified}  
 if and only if the following conditions are satisfied:
\begin{enumerate}
\item The algebra $A$ is an iterated inflation of cellular algebras $B_l$ along vector spaces $V_l$ for $0\leq l \leq n$.
\item For each $l\leq n$ there exists a non-zero element $\epsilon_l  \in V_l$ such that
\begin{align*}
e_l =  \epsilon_l \otimes \epsilon_l \otimes 1_{B_l},\end{align*}is an idempotent.
\item If $l \leq m$, then $e_le_m=e_m = e_m e_l$.
\end{enumerate}
 \end{defn}
For brevity we have assumed that the idempotents are fixed by the anti-involution.  A few degenerate cases   require minor modifications (see \cite[Section 5.2]{HHKP}).

The idempotents from the definition of a cellularly stratified algebra give rise to a chain of two-sided idempotent  ideals $A=J_0 \supseteq J_1 \supseteq \ldots \supseteq J_n=\{0\}$, where $J_l = Ae_lA$.    This chain of ideals provides a \emph{stratification} (in the sense of \cite{CPS}) of the algebra $A$ (see \cite[Proposition 7.2]{HHKP}).  
 
 It is well-known that a cellular algebra, $A$, is quasi-hereditary if and only if the chain of cell ideals gives a stratification of $A$.  Therefore cellularly stratified algebras can be seen as an intermediate step,  
 between cellular and quasi-hereditary algebras.

\subsection{Inductive functors for cellularly stratified algebras}

Cellularly stratified algebras provide an effective framework for the study of a diagram algebra in terms of smaller input algebras.  
 The following lemma allows us to think of the input algebras as subquotients of the diagram algebra.
 
\begin{lem}[Section 2 \cite{HHKP}]\label{HHKP1}
Let $A$ be cellularly stratified.  There is an isomorphism $B_l \simeq e_lAe_l/e_lJ_{l-1}e_l$ with $1_{B_l}$ mapped to $e_l$.
\end{lem}

\subsubsection{Globalisation and localisation}\label{inflation}   Following \cite{HHKP} we define globalisation and localisation functors between $e_lAe_l$ and $A$ as follows:
\begin{align*}
\G_l^A	&: \text{mod-}e_lAe_l \to \text{mod-}A			\\		
	&:M 		\mapsto	(A/J_{l+1})e_l \otimes_{e_l A e_l} M,		\\
\F_l^A	&: \text{mod-}A \to \text{mod-}e_lAe_l		\\
	&:N 		\mapsto	e_l N.
\end{align*}
 \begin{prop}[Section 4 \cite{HHKP}]\label{HHKP1}The globalisation functor has the following properties:
\begin{itemize}
\item The functor $\G_l^A$ is exact
\item Let $M$ be any $B_l$-module.  Then $\G_l^A(M) = V_l \otimes M$ as a vector space.
\end{itemize}
Let $M$ and $N$ be $B_l$-modules.  Then $\Hom_A(\G_l^A(M),\G_l^A(N))\simeq\Hom_{B_l}(M,N)$.
\end{prop}
 Note that when we write $\otimes$ without a subscript, this tensor product should be taken to be over $K$.
The following lemma is a generalisation of \cite[Lemma 11]{Row} to the setting of cellularly stratified algebras.
We may globalise modules to any subalgebra of the form $e_mAe_m$.  We adjust the superscripts by letting  $\G_l^m=\G^{e_mAe_m}_l$ and $\F_l^m=\F^{e_mAe_m}_l$. 
\begin{lem}\label{restrict}
Let $M$ be a $B_l$-module.  Then we have that 
\begin{align*}
\F^A_m (\G_l^AM)	\simeq 
\begin{cases}
\G_l^{m}	M &\text{ if }l \geq m \\
0				& \text{ otherwise}
\end{cases}.
\end{align*}
\end{lem}

\begin{proof}
From the definitions of the functors we have that:
\begin{align*}
\F^A_m (\G_l^AM)	&=		e_m(A/J_{l+1})e_l \otimes_{e_l A e_l} M.		 
\end{align*}
We note that $e_le_m=e_me_l = e_l$ if $l \geq m$, and therefore in this case we can multiply on the right of $e_m(A/J_{l+1})e_l$ by $e_m$ without effect.  If $l < m$, then left multiplication by $e_m$ annihilates $A/J_{l+1}$ and so we get
\begin{align*}
=				\begin{cases}
					((e_mAe_m) /e_mJ_{l+1}e_m)e_l \otimes_{e_l A e_l}    M \text{ if $l \geq m$},\\
					0 \text{ otherwise.} 
					\end{cases}	
\end{align*}\end{proof}

\subsubsection{Induction and Restriction}\label{ind}  
The following induction functors are used in \cite{Row} in order to define permutation modules for the Brauer algebra.  
\begin{align*}
\Ind_{e_lAe_l}^A	&: \text{mod-}{e_lAe_l} \to \text{mod-}A			\\
	&:M 		\mapsto	Ae_l \otimes_{e_l A e_l} M,		\\
\Res_{e_lAe_l}^A	&: \text{mod-}A \to \text{mod-}{e_lAe_l}	\\
	&:N 		\mapsto	e_lN.
\end{align*}
We compose these functors with the usual induction and restriction to obtain induction and restriction between $A$ and any unitary subalgebra, ${H}$, of $e_lAe_l$.
\begin{align*}
\Ind_{{H}}^A	&: \text{mod-}{H} \to \text{mod-}A			\\
	&:M 		\mapsto	Ae_l \otimes_{{H}} M,		\\
\Res_{{H}}^A	&: \text{mod-}A \to \text{mod-}{H}	\\
	&:N 		\mapsto	e_lN.
\end{align*}
We have the following adjunctions:
\begin{align*}
\Hom_A(\Ind_{{H}}^A(M), N) \!\simeq\! \Hom_A({{Ae_l}}\otimes_HM, N)\!\simeq\!\Hom_{{H}}(M, \Hom_A(Ae_l,N)) 
\!\simeq\!\Hom_{{H}}(M, e_lN) . 
\end{align*}

\subsection{A filtration of the induction functor}\label{wobbles} 
 
 In order to generalise the  Dipper--James--Green  basis of an endomorphism algebra, one requires an exact induction.  Our induction functor is readily seen to be right exact only, and so we need to construct a filtration of this induction.  

The induction functor  
arises from tensoring with the one-sided ideal $Ae_l$.  We recall the filtration of $A$ by two-sided ideals $A = J_0 \supseteq J_1 \supseteq \ldots \supseteq J_n \supseteq 0,$ where $J_i=Ae_iA$.  Therefore $Ae_l$ has a filtration by left ideals $Ae_l = J_{l,0}  \supseteq J_{l,1} \supseteq \ldots \supseteq 0,$ where $J_{l,i}=(Ae_{l+i}A/Ae_{l+i+1}A)e_l$.
Therefore our induction has a filtration
 with subquotients given by $$J_{l,i}/J_{l,i+1} \otimes_{e_lAe_l} - : e_lAe_l\text{-mod} \to A\text{-mod}.$$
 Note that the first layer $J_{l,0}/J_{l,1} \otimes_{e_lAe_l}M\simeq\G_l(M)$ and so the top layer of the induction functor is isomorphic to the (exact) globalisation  functor.
In the cases of the Brauer, walled Brauer, and partition algebras we shall exhibit each layer as a composition of (exact) globalisation and localisation functors.

\section{Quasi-hereditary Covers} We review the notion of a quasi-hereditary cover due to Rouquier.  We then discuss how we will construct covers of diagram algebras.
   \subsection{Quasi-hereditary covers} 
Let $S$ be a quasi-hereditary algebra and $M$ be a finitely generated projective module.  Let $A=\End_S(M)$.  We are interested in the following double centraliser property. 
\begin{defn}[Rouquier]
We say that  the pair $(S,M)$ is a \emph{quasi-hereditary cover} of $A$ if the restriction of $F=\Hom_S(M, - )$ to the category of projective modules for $S$ is fully faithful.
\end{defn}
We want to know the strength of the connection between an algebra and its cover. This is achieved by considering the level of `faithfulness' in  cohomology.
\begin{defn}[Rouquier]
Let $i$ be a non-negative integer.  We say that the pair $(S,M)$ is an $i$-\emph{faithful cover} of $A$ if $F=\Hom_S(M, -)$ induces isomorphisms $\Ext^j_S(M,N)$ $\simeq \Ext^j_A(FM,FN)$ for all $M, N \in \mathcal{F}(\Delta)$ and $j\leq i$.
\end{defn}
\noindent 
It was shown in \cite[Proposition 4.45]{rouq} that 1-covers are unique up to Morita equivalence.

\subsection{Covers of diagram algebras}
Let $A$ be cellularly stratified, with input algebras $B_l$.  We have that $B_l$ appears as a quotient, $e_l(A/Ae_{l+1}A)e_l$, of $e_lAe_l$.  Assume further that $B_l$ appears as a unitary subalgebra of $e_lAe_l$, for all $l\leq n$.  This is a natural assumption as can be seen in examples such as the  walled Brauer algebra (see \cite{coxwall}),  and the   Brauer and partition algebras (see \cite{HHKP}) and their cyclotomic analogues.

Assume the input algebras, $B_l$, have 
integrally defined quasi-hereditary covers $(S(B_l),$ $M_{B_l})$.  Following Hartmann and Paget we define the permutation modules, $M_l$, for the diagram algebra, $A$, to be the induced modules, $M_l=\Ind_{B_l}^AM_{B_l}$, from the subalgebras $B_l$ of $e_lAe_l$.
  These modules are characteristic and specialisation-free.  We let $M=\oplus_{l\geq0} M_l$.  We shall study the endomorphism algebra
\begin{align*}
S(A)=\End_A(M).
\end{align*}

In our examples we shall consider diagram algebras with a stratification given by group algebras of symmetric groups.
Therefore the module $M$ is a direct sum of `Young' permutation modules $M=\oplus_{l\geq 0, {\lambda \in \Lambda_A}}M(\lambda,l)$ where $M(\lambda,l)=\Ind^A_{\Sigma_{\lambda}}K$.

We recast \cite[Theorem 13.1]{HHKP} in terms of permutation modules.


\begin{thm}\label{HHKPmain}
Let $A$ be a cellularly stratified algebra, and suppose that $e_lAe_l$ has a subalgebra isomorphic to the input algebra $B_l$  for all $l\leq n$.  Suppose that for each $l$ the algebra $B_l$ has a 1-faithful cover $(S({B_l}),M_l)$. 
If $\Hom_A( M
, -)$ is exact on $\mathcal{F}(\Delta)$
, then
(S(A),M) 
is a 1-faithful quasi-hereditary cover of $A$. 
Moreover $S(A)$ is Morita equivalent to the Schur algebra constructed in \cite{HHKP}.
\end{thm}
\begin{proof}
The existence of such a cover is proved in \cite{HHKP}.  That $M_l$ 
 is a direct sum of `Young modules'  (in the sense of \cite[Definition 11.2]{HHKP}) follows from \cite[Lemma 22]{Row} and our assumption that $\Hom_A(M, -)$ is exact on $\mathcal{F}(\Delta)$. 

 Let $\{Y(\lambda)\}_{\lambda \in \Lambda_{B_l}}$ be the `Young modules' for the algebra $B_l$.  These all appear as direct summands of $M_l$ by assumption.  We therefore get that $Y(\lambda,l)=\Ind_{B_l}^A(Y(\lambda))$ appear as direct summands of $M$.  By the previous paragraph $Y(\lambda,l)$ is a direct sum of `Young modules' and by properties of the induction there exists a projection $Y(\lambda,l) \to \Delta(\lambda)$.  Therefore for all $(\lambda,l) \in \Lambda_A$ we get that $Y(\lambda,l)$ is a direct summand of $\Ind_{B_l}^AM_{B_l}$.

Therefore our definition of $S(A)$ is Morita equivalent to that in \cite{HHKP} and therefore $S(A)$ is a 1-faithful cover. 
\end{proof}

 \section{Bases of the Schur Algebra}
We wish to construct a cellular basis of the cover $(S(A),M)$ by lifting the cellular bases from the covers of the input algebras $(S(B_l),M_l)$.   
All the examples of diagram algebras we shall consider are iterated inflations of symmetric groups and so we review the combinatorics of the classical Schur algebra. 
 
\subsection{The Green--Dipper--James basis}

We review the combinatorics and construction of this basis (see \cite{Green} and \cite{DJ} for the quantised version), as it will be essential for what follows.  

\subsubsection{}\label{distinguish}
Let $\Sigma_r$ denote the symmetric group on $r$ letters.  A \emph{composition} $\lambda$ of $r$ is a sequence $(\lambda^1, \lambda^2,  \ldots)$ such that $\sum_{i=1}^\infty \lambda^i=r$; this will be denoted $\lambda \vDash r$.  If in addition this sequence is weakly decreasing then it is called a \emph{partition} and will be denoted $\lambda \vdash r$.

We let $\Std(\lambda)$ denote the set of standard $\lambda$-tableaux.  For $\lambda \vDash r$ we let $\mathfrak{t}^\lambda$ denote the $\lambda$-tableau in which the numbers $1, \ldots , r$ appear along successive rows.  We have a natural left action of $\Sigma_r$ on the set of all $\lambda$-tableaux by  
letter permutations.  
For example, for $\lambda = (3,2)$, $w=(3 \ 5)  \in \Sigma_5$,
\begin{align*}
\begin{minipage}{44mm}
\begin{equation*}  
w\mathfrak{t}^\lambda   =
\def\objectstyle{\scriptstyle}
\xymatrix@=6pt{
&  1 w  &2  w & 3 w  \\
& 4 w  & 5  w
}
\end{equation*}
\end{minipage}
=
\begin{minipage}{30mm}
\begin{equation*}  
\def\objectstyle{\scriptstyle}
\xymatrix@=6pt{
  1    &2    & 5    \\
 4    & 3   
}
\end{equation*}
\end{minipage}
\end{align*}

We let $x_\lambda$ (respectively $y_\lambda$) denote the sum  (respectively signed sum) over the row  (respectively column) stabiliser of $\mathfrak{t}^\lambda$.  The Specht module, $\mathcal{S}(\lambda)$, is defined to be $K\Sigma_r x_\lambda y_\lambda$.

For $\mu \vDash r$ we let $\mathcal{D}
_\mu=\{w \in \Sigma_r : w t^\mu    \text{ is row-standard} \}$.  It is well known that $\mathcal{D}_\mu $ forms a set of left coset representatives of $\Sigma_\mu$ in $\Sigma_r$.  For $\lambda \vDash r$ we analogously define $\mathcal{D}_\lambda^{-1} $ by $\Sigma_r$ acting on the right (this action is given by the above, composed with inversion); this gives a   set of right coset representatives of $\Sigma_\lambda$ in $\Sigma_r$.

Therefore each row standard $\mu$-tableau corresponds to a coset of $\Sigma_{\mu}$ in $\Sigma_r$.  If $\mathfrak{t}$ is a row standard $\mu$-tableau let $d(\mathfrak{t})$ be the unique element of $\Sigma_r$ such that $\mathfrak{t} = d(\mathfrak{t}) \mathfrak{t}^\mu$.
We have that $\mathcal{D}_{ \mu \lambda} = \mathcal{D}_\mu \cap \mathcal{D}_\lambda^{-1}$ is a distinguished set of $\Sigma_\lambda-\Sigma_\mu$ double cosets in $\Sigma_r$, which we have defined through pairs of row-standard tableaux.

 Let $\lambda$ and $\mu$ be  partitions of $n$.  We say that $\lambda$
dominates $\mu$ and write $\mu \unlhd_n \lambda$ if
$  \sum_{i=1}^k \lambda^i  \geq  
\sum_{i=1}^k \mu^i 
$
for all $k\geq 0$.
 
\subsubsection{}
 
Let $\lambda=(\lambda^1,\lambda^2, \ldots) \vDash r$.  We then define the Young subgroup $\Sigma_\lambda \leq \Sigma_r$ to be the direct product $\Sigma_{\lambda^1} \times \Sigma_{\lambda^2} \times \ldots \leq \Sigma_r$.  We define $M(\lambda)$ be the Young permutation module $K\!\!\uparrow_{\Sigma_\lambda}^{\Sigma_r}$.
 The following theorem gives a basis for the homomorphisms between Young permutation modules.   

\begin{thm}[The Green--Dipper--James Basis]\label{DJbasisofschur}
Let $\lambda, \mu \vDash r$, 
then $\{\varphi_d : d \in \mathcal{D}_{\lambda\mu  }\}$ is a basis of $\Hom_{K\Sigma_r}(M(\lambda),M(\mu))$, where $\varphi_d$ is given by:
\begin{align*}
\varphi_d(x_\lambda) &= \sum_{w \in \mathcal{D}_\nu \cap \Sigma_\lambda}
wdx_\mu , 
\end{align*}
where $\nu$ is the composition of $r$ which labels the Young subgroup $d\Sigma_\mu d^{-1}  \ \cap \ \Sigma_\lambda$ of $\Sigma_r$.  \end{thm}
\begin{Remark on the proof}
 The theorem is proved by the application of both  {Frobenius reciprocities} and {Mackey decomposition}.  
  \end{Remark on the proof}

 \subsection{The semistandard basis theorem}
The semistandard basis theorem  gives a cellular basis of the Schur algebra.  We follow the construction of this basis due to Dipper, James, and Mathas (see \cite{DJM})   given by `lifting' the Murphy basis of the permutation modules of the symmetric group. 
\subsubsection{}
Suppose $\omega$ is a partition of $r$ and let $\mathfrak{t}$ be an $\omega$-tableaux. Say that $\mathfrak{t}$ is of type $\lambda$ if each integer $i \geq 1$ appears $\lambda^i$ times in $\mathfrak{t}$. Let $T (\omega, \lambda)$ denote the set of $\omega$-tableaux of type $\lambda$. We say that $\S \in T  (\omega, \lambda)$ is row-standard if the entries are non-decreasing along the rows, and semistandard if it is row-standard and the entries are strictly increasing down the columns.  We let $T_0(\omega,\lambda)$ denote the semistandard $\omega$-tableaux of type $\lambda$ and $T_0(\lambda)$ denote the union $\bigcup_{\omega \vdash  r}T_0(\omega, \lambda)$.

We wish to convert tableaux of type $\nu$ to tableaux of type $\omega$.  Let $\mathfrak{t}$ be a tableau of type $\nu$ and let $\omega$ be a composition.  We define $\omega(\mathfrak{t})$ to be the tableau of type $\omega$ obtained from $\mathfrak{t}$ by replacing the entry $i$ in $\mathfrak{t}$ by $r$ if $i$ appears in row $r$ of $\mathfrak{t}^\omega$.

\subsubsection{}\label{Smurphet} Let $\omega \vdash r$ and $\lambda\vDash r$.  Let $\T$ be a semistandard $\omega$-tableaux of type $\lambda$ and let $\mathfrak{s}$ be a standard $\omega$-tableau.   Define $m_{\mathfrak{s} \mathfrak{t}}= {d(\mathfrak{s})}x_\mu d({\mathfrak{t}})^{-1}$.  Then $m_{\mathfrak{s} \mathfrak{t}}^\ast = m_{\mathfrak{t} \mathfrak{s}}$.
We define
\begin{align*}
m_{ \mathfrak{s}\T} = \sum_{	\begin{subarray}{l}
        \mathfrak{t}\in \Std(\omega), \\
	\lambda(\mathfrak{t})=\T
	      \end{subarray}		  }
m_{\mathfrak{s}\mathfrak{t}}
\end{align*}
and let $m_{\T\mathfrak{s}} = m_{\mathfrak{s}\T}^\ast$.
\begin{thm}[Murphy basis \cite{16}]
The module $M(\lambda)$ is free as a $K$-module with basis
\begin{align*}
\{ m_{ \mathfrak{s}\T} : \T \in T_0(\omega, \lambda), \mathfrak{s} \in \Std(\omega) \text{ for some } \omega \vdash r\}.
\end{align*}
\end{thm}

\begin{cor}
The $\Sigma_r$-module $M(\lambda)$ has a filtration $M(\lambda)=M_1 \supseteq M_2 \supseteq \ldots \supseteq M_{k+1} =0$ such that there exists $\mu^i\vdash r$ with $M_i/M_{i+1} \simeq \mathcal{S}(\mu^i)$. Moreover, for each partition $\mu$ the number of $\mu^i$ equal to $\mu$ is the number of semistandard $\mu$-tableaux of type $\lambda$.
\end{cor}

\subsubsection{}\label{homotheoremgreen} We recall the notation needed to state the semistandard basis theorem.  Let $\omega\vdash r$ and let $\lambda,\mu \vDash r$.  Suppose $\S \in T_0(\omega,\lambda )$ and $\T \in T_0(\omega, \mu)$.  Now define
\begin{align*}
m_{\S \T} = \sum_{\mathfrak{s}, \mathfrak{t}} m_{\mathfrak{s}\mathfrak{t}},
\end{align*}
where the sum is over all pairs $(\mathfrak{s},\mathfrak{t})$ of standard $\omega$-tableaux such that $\lambda(\mathfrak{s}) = \S$ and $\mu(\mathfrak{t})= \T$.  A basis of $\Hom_{\Sigma_r}(M(\lambda), M(\mu))$ is given by the maps:
\begin{align*}
\varphi_{\S\T} :& M(\lambda) \to M(\mu) \\
			& x_\lambda \longmapsto m_{\S\T},
\end{align*}
which are indexed by pairs of semistandard tableaux $\S \in  T_0(\omega,\lambda )$ and $\T \in T_0(\omega, \mu)$.  

Define the Schur algebra $S(r)$ to be the endomorphism algebra  $\End_{\Sigma_r}(\oplus_{\lambda \vdash r} M(\lambda))$.   
We trivially extend the domain of these homomorphisms to be elements of $S(r)$ and define the ideal $S^\omega(r)$ to be the $K$-module spanned by $\varphi_{\S\T}$ such that $\S, \T \in T_0(\alpha)$ for some $\alpha \rhd \omega$. 
For the proof of the following theorem (in the full generality of cyclotomic Hecke and $q$-Schur algebras) we refer to \cite{DJM}.
\begin{thm}\label{cellbasisofschur}
The Schur algebra $S(r)$ is a free $K$-module with basis:
\begin{align*}
\{ \varphi_{\S\T} : \omega \vdash r \text{ and } \S, \T \in T_0(\omega)\}.
\end{align*}
Moreover:
\begin{itemize}
\item The map $\ast : \varphi_{\S\T} \to \varphi_{\T\S}$ is an anti-automorphism of $S(r)$.
\item Suppose that $\omega \vdash r $ and that $\S$ is a semistandard $\omega$-tableau.  Then for all $\varphi  \in S(r)$ there exist $k_{\V} \in K$ such that for all $\T \in T_0(\omega)$
\begin{align*}
\varphi_{\S\T}\circ \varphi  =\sum_{\V \in T_0 (\lambda)}k_\V \varphi_{\S\V} \quad \text{\rm{mod}-}S^\omega(r).
\end{align*}
\end{itemize}
Consequently, this is a cellular basis of the Schur algebra.
\end{thm}

\section{The Brauer algebra}

In this section we give an example-led review of the representation theory of the Brauer algebra, $B_K(r,\delta)$ (sometimes denoted $B_r$).  
 Our introduction is based on \cite{Row}. 

We introduce modified tableaux and use these to derive the Murphy basis and `Specht' filtration of  a permutation module and to give bases of their endomorphism algebras.  
\subsection{Definitions and examples}\label{EXAMbrauer}
Let $A=B_K(r,\delta)$ be the classical Brauer algebra. The algebra, $B_K(r,\delta)$, has as a $K$-basis the set of all diagrams consisting of a row of $r$ \emph{northern} vertices and a row of $r$ \emph{southern} vertices,  with each vertex joined to exactly one other vertex by an edge.  
A vertical edge identifying a northern vertex to a southern vertex is called a \emph{through-line} and a horizontal edge is called an \emph{arc}.

Multiplication of two diagrams $x$ and $y$ is defined by concatenation; the bottom row of $x$ is identified with the top row of $y$, following the edges from a vertex on the top row of $x$ to the bottom row of $y$ identifies a new basis element $z$.  We let $j$ denote the number of closed loops in the middle.  We then define the multiplication by $x \cdot  y = \delta^j z$.  This defines a generically semisimple algebra over the complex numbers.  Non-semisimple representations can occur over modular fields or upon specialisation of  $\delta$ to an integer.
 
  In \cite{HHKP} the Brauer algebra was shown to be cellularly stratified with inflation decomposition,
  \begin{align*}
B_K(r,\delta) = \oplus_{l} V_l \otimes V_l \otimes \Sigma_{r-2l}.
\end{align*}
where $V_l$ is the vector space of all possible configurations of $l$ arcs, and the symmetric groups provide the through-lines.
If $\delta \neq 0$ then we define $e_l$ to be $1/\delta^l$ times the element
\begin{align*}
\begin{minipage}{54mm}
\def\objectstyle{\scriptstyle}
\xymatrix@=1pt{
&& \circ 			&&\cdots	&&\circ  			&&\circ 		&&\circ  \ar@{-}[ll]	&&\cdots	&&\circ 	 	&&\circ\ar@{-}[ll]  \\	&\\
&&\circ \ar@{-}[uu] 	&&\cdots	&&\circ  \ar@{-}[uu] 	&&\circ 		&&\circ  \ar@{-}[ll]	&&\cdots	&&\circ 	 	&&\circ \ar@{-}[ll]
}
\end{minipage}&, 
\end{align*}
with $r-2l$ straight through-lines and $l$ arcs, each joining two consecutive nodes of the final $2l$ nodes.  If $\delta=0$ and $l\neq r/2$, then one can define alternative idempotents which satisfy the necessary conditions (see \cite[Section 2.2]{HHKP}).

\begin{prop}[Proposition 1.1 of \cite{DT}]\label{A filtration of induction for Brauer}
   Let $r=t+2l=s+2m$, let $m-l=i\in \mathbb{Z}$, and assume without loss of generality that $i\geq0$. 
   Let ${^mV^l}$ denote the $K\Sigma_{s}-K\Sigma_{t}$-bimodule  $e_m(A/Ae_{l+1}A)e_l$.
We have that 
\begin{align*}
{^lV^m}\otimes_{\Sigma_{s}} - : \Sigma_{s}{\text{\rm -mod}} \xrightarrow{\rm inflation} &\Sigma_{s} \times \Sigma_2 \wr \Sigma_i {\text{\rm -mod}} \xrightarrow{\rm induction}\Sigma_{t}{\text{\rm -mod}} \\
M \longmapsto &M \boxtimes K \longmapsto (M \boxtimes K)\! \uparrow_{\Sigma_{s} \times \Sigma_2 \wr \Sigma_i }^{\Sigma_{t}}
\end{align*}
where the first functor is the standard inflation from the quotient of the group $\Sigma_{t} \times \Sigma_2 \wr \Sigma_i$ by the normal subgroup $\Sigma_2 \wr \Sigma_i$.  Similarly,  we have that
\begin{align*}
{^mV^l}\otimes_{\Sigma_{t}} - : \Sigma_{t}{\text{\rm -mod}}& \xrightarrow{\rm restriction} \Sigma_{s} \times \Sigma_2 \wr \Sigma_i {\text{\rm -mod}} \xrightarrow{\rm projection}\Sigma_{s}{\text{\rm -mod}} \\
&N \longmapsto N\!\! \downarrow_{\Sigma_{s} \times \Sigma_2 \wr \Sigma_i} \longmapsto N/\{ n - hn: h \in \Sigma_2 \wr\Sigma_i\}
\end{align*}and we shall let $e_{l+i}N$ denote the image of $N$ under the localisation map.  
Furthermore, $$J_{l,i}/J_{l,i+1}\otimes_{K\Sigma_{t}} M \simeq J_{l+i,0}/J_{l+i,1}\otimes_{K\Sigma_{t-2i}}(e_iM).$$ 
\end{prop}
\begin{proof}
The left-action of $\Sigma_t$ on ${^lV^m}$ is given by the transitive permutation of the first $t$ nodes in a diagram, with stabiliser $\Sigma_s \times \Sigma_2 \wr \Sigma_i$; as a right $\Sigma_s$-module, ${^lV^m}$ is isomorphic to the group algebra $\Sigma_s$.   Therefore, the globalisation functor obtained by the tensor product with this bi-module is given by inflation from $\Sigma_s$ to $\Sigma_s \times\Sigma_2 \wr \Sigma_i$ composed with induction to $\Sigma_t$.  The adjoint functor is constructed similarly.

  
We now consider the filtration of the induction functor. The subquotient $ J_{l,i}/J_{l,i+1}$ is spanned by all diagrams with southern row that consists of $l$ arcs each joining two consecutive of the last $2l$ nodes, and exactly $i$ other arcs.  Any $u \otimes v \otimes b \in J_{l,i}/J_{l,i+1}$ can be rewritten as the product $\delta^{-(l+i)}(u \otimes \epsilon_{l+i} \otimes b)(\epsilon_{l+i} \otimes v \otimes  \id_{K\Sigma_{t-2i}})$; note that the choice of $u \otimes \epsilon_{l+i} \otimes b\in J_{l+i,0}/J_{l+i,1}$ and $\epsilon_{l+i} \otimes v \otimes  \id_{K\Sigma_{t-2i}}\in {^{l+i}V^l}$ is unique up to a permutation of the  $t-2i$ through-lines, i.e,
$J_{l,i}/J_{l,i+1}\otimes_{K\Sigma_{t}} M \simeq J_{l+i,0}/J_{l+i,1}\otimes_{K\Sigma_{t-2i}}(e_{l+i}M).$
 \end{proof}
\begin{eg}
We consider as an example $B_K(3,\delta)$ over a field of characteristic zero.  For $B_K(3,\delta)$ there are the three conjugacy classes represented by elements of the symmetric group and two classes represented by elements with one arc, one where the top and bottom rows match, and one where they do not.
\begin{align*}
x=\begin{minipage}{34mm}
\def\objectstyle{\scriptstyle}
\xymatrix@=1pt{
\circ 		&&\circ 	\ar@{-}[ll]	&&\circ\ar@{-}[dd]	&\\
&\\
\circ 	\ar@{-}[rr]	&&\circ 		&&\circ
}
\end{minipage} \sim 
\begin{minipage}{34mm}
\def\objectstyle{\scriptstyle}
\xymatrix@=1pt{
&\circ 	\ar@{-}[dd]	&&\circ 	\ar@{-}[rr]	&&\circ	&\\ &\\
&\circ 		&&\circ 	\ar@{-}[rr]	&&\circ
}
\end{minipage},  \quad \quad
y=\begin{minipage}{34mm}
\def\objectstyle{\scriptstyle}
\xymatrix@=1pt{
&\circ 	&&\circ 	\ar@{-}[ll]		&&\circ\ar@{-}[ddllll]	&\\ &\\
&\circ 	&&\circ 	\ar@{-}[rr]	&&\circ
}
\end{minipage} \sim
\begin{minipage}{34mm}
\def\objectstyle{\scriptstyle}
\xymatrix@=1pt{
&\circ \ar@{-}[ddrrrr]		&&\circ 	\ar@{-}[rr]	&&\circ	\\
&\\&\circ 	\ar@{-}[rr]	&&\circ 		&&\circ
}
\end{minipage} ,
\end{align*}
Note that there are three elements conjugate to $x$ and six elements conjugate to $y$.  The standard module $\Delta((1),1)$ has basis $$
\left\{
\begin{minipage}{34mm}
\def\objectstyle{\scriptstyle}
\xymatrix@=1pt{
\circ \ar@{-}[rr]	&&\circ 	&&\circ	\\ & \\
\circ \ar@{-}[uurrrr]	&&\circ  \ar@{-}[rr]		&&\circ	
}
\end{minipage} ,
\begin{minipage}{34mm}
\def\objectstyle{\scriptstyle}
\xymatrix@=1pt{
\circ &&\circ 	&&\circ	\ar@{-}[ll]	\\ & \\
\circ \ar@{-}[uu]	&&\circ  \ar@{-}[rr]		&&\circ	
}
\end{minipage} ,
\begin{minipage}{34mm}
\def\objectstyle{\scriptstyle}
\xymatrix@=1pt{
\circ  \ar@{-}@/_.8pc/[rrrr]	&&\circ 	&&\circ\\ & \\
\circ \ar@{-}[uurr]	&&\circ  \ar@{-}[rr]		&&\circ		
}
\end{minipage}\right\}.$$ 
The quotient $A/Ae_1A\simeq \Sigma_3$ acts by permuting these basis elements.  The elements $x$ and $y$ act as follows:
\begin{align*}
x=\left(\begin{array}{ccc}\delta & 1 & 1 \\0 & 0 & 0 \\0 & 0 & 0\end{array}\right) \text{ and } y= \left(\begin{array}{ccc}1 & \delta & 1 \\0 & 0 & 0 \\0 & 0 & 0\end{array}\right).
\end{align*}
Therefore the character table (see \cite{RAM}) for the generic algebra is as follows:
\begin{align*}
\begin{array}{c|ccccc} & e & (12)	&(123) & \	x	\ & \ y\			
 \\ \hline \Delta((3),0) & 1 & 1 & 1&0&0 \\ 
 \Delta((2,1),0) & 2 & 0 &-1&0&0 \\ 
 \Delta((1^3),0) & 1 & -1&1&0 & 0\\
  \Delta((1),1) & 3 & 1 &0&\delta&1  \end{array}.
\end{align*}

From the matrix representations of $x$ and $y$ we deduce that upon specialisation of $\delta$ to $-2$ that $x$ and $y$ act as zero on $\Span_K\{v_1 + v_2 +v_3\}$.  This gives rise to an embedding of the trivial module $ \Delta((3),0)\hookrightarrow  \Delta((1),1) $.  

Similarly, upon specialisation of $\delta$ to $1$ we have that $x$ and $y$ act as zero on the subspace $\Span_K\{v_1-v_2,v_2-v_3\}$.  This gives rise to an embedding $ \Delta((2,1),0)\hookrightarrow  \Delta((1),1) $.
These are the only non-semisimple specialisations in characteristic zero. 

We have that $M((1),1)=B_K(3,\delta)e_1\otimes_{K\Sigma_1}K \simeq \Delta((1),1)$ by the above.  We can now calculate the character of $M((3),0)$ using Frobenius reciprocity:
\begin{align*}
\Hom_{B_K(3,\delta)}(M((3),0),M((1),1))	
&\simeq\Hom_{\Sigma_3}(\mathcal{S}(3), \Res_{K\Sigma_3}^{B_K(3,\delta)} \Delta((1),1),\\
&\simeq\Hom_{\Sigma_3}(\mathcal{S}(3),\mathcal{S}(3) \oplus \mathcal{S}(2,1) ) 
\simeq K.
\end{align*}

The second equality was deduced from the character table.  Therefore as a generic $B_K(3,\delta)$-module $M((3),0)\simeq\Delta((3),0) \oplus \Delta((1),1)$.   There is therefore a unique homomorphism from $M((3),0)$ to $M((1),1)$.   In Section \ref{frob1} we shall describe this homomorphism explicitly in terms of basis elements.

We wish to describe the structure of the module $M((3),0)$, in the case that $\delta=-2$.
We have that $\mathcal{S}(3)$ is a projective $\Sigma_3$-module, and so $M((3),0)$ is a projective $B_\mathbb{C}(3,-2)$-module.  
We have already seen that $0\to \Delta((1),1) \to M((3),0) \to\Delta((3),0) \to 0$. 
As $M((3),0)$ is projective, we conclude that the sequence is non-split and that $M((3),0)$ is the uniserial module $[L((1),1), L((3),0), L((1),1)]$.


\end{eg}

\subsection{Modified tableaux}\label{mod brauer}
In the case of the Brauer algebra we shall always take $r= t+ 2l= s+2m$, and $\lambda \vdash t$, $\mu \vdash s$.  
 In this section we shall only need to consider the case of $m-l \geq 0$ (as homomorphisms occur from higher layers to lower layers by \cite[Corollary 7.4]{HHKP}).  We let $i=m-l$.

In the case of the Brauer algebra we shall need to index homomorphisms using a set of double coset representatives, $
\mathcal{D}_{\lambda\downarrow,\mu}$, of  $\Sigma_{\lambda} \backslash \Sigma_t / \Sigma_{\mu} \times \Sigma_2\wr\Sigma_i$.

To define the image of a homomorphism we will require a set of left coset representatives for $\Sigma_t / \Sigma_\nu$, where $\nu\vDash r$ corresponds to the subgroup $\Sigma_\lambda  \ \cap \ d(\Sigma_\mu    \times \Sigma_2\wr\Sigma_i)d^{-1} $ of $\Sigma_t$ for $d \in \mathcal{D}_{\lambda\downarrow,\mu}$.  This group will be a direct product of a Young subgroup and a series of hyperoctahedral groups.

A composition $\lambda \vDash t$ can be illustrated by a \emph{Young diagram}, $[\lambda]$, consisting of $t$ nodes  placed in rows.  The $j^{th}$ row of $[\lambda]$ consists of $\lambda^j$ nodes, and all the rows start in the same column.  

A \emph{modified diagram} arises from a pair  $(\lambda, i)$ by taking the Young diagram $[\lambda]$ and connecting $i$ pairs of nodes from the diagram, we shall also call these \emph{$(\lambda,i)$-diagrams}. 

Define an equivalence relation on modified diagrams by neglecting the order of the points in the row.  
A \emph{row-standard modified diagram} is defined to be any representative of an equivalence class under this relation.
We shall refer to the set of all diagrams in an equivalence class of the diagram $\sigma$ as the \emph{coset} of $\sigma$.  We shall let $\{\sigma\} \in \sigma$ denote an element of the coset and let $[\sigma]$ denote the sum over all elements of the coset.
We let ${T}^i_\lambda$ denote the set of all row-standard $(\lambda,i)$-diagrams (we sometimes suppress the $i$ and write ${T}^i_\lambda$).  For example,
\begin{align*}
\left\{
\begin{minipage}{34mm}
\def\objectstyle{\scriptstyle}
\xymatrix@=6pt{
   \ast  &	\ast	 	\\
\ast \ar@{-}[u] 	}
\end{minipage} 
\right\} + \left\{
\begin{minipage}{34mm}
\def\objectstyle{\scriptstyle}
\xymatrix@=6pt{
   \ast  &	\ast		\\
\ast \ar@{-}[ur] 	}
\end{minipage} 
\right\}
\ = \ 
\left[\begin{minipage}{34mm}
\def\objectstyle{\scriptstyle}
\xymatrix@=6pt{
 \ast  &	\ast		\\
 \ast \ar@{-}[u]    	 
}
\end{minipage} \right]
	\quad	\text{ and } \quad 
\left\{
\begin{minipage}{34mm}
\def\objectstyle{\scriptstyle}
\xymatrix@=6pt{
 \ast \ar@{-}[r]  &	\ast		\\
 \ast    	 
} \end{minipage}
\right\} =
\left[\begin{minipage}{34mm}
\def\objectstyle{\scriptstyle}
\xymatrix@=6pt{
 \ast \ar@{-}[r]  &	\ast		\\
 \ast    	 
} \end{minipage}\right]  
\end{align*}
are the coset sums of the two elements of $T^1_{(2,1)}$.  

A $(\lambda, i)$-\emph{tableau} arises from a row-standard $(\lambda, i)$-diagram by replacing the unconnected nodes with the numbers $\{1, \ldots, t-2i\}$.  For example\begin{align*}
\begin{minipage}{64mm}
\def\objectstyle{\scriptstyle}
\xymatrix@=6pt{
 1  & 2  &	\ast	&&	\\
 3  &\ast \ar@{-}[ur] 		&&	\\
 4			}
\end{minipage} 
&\begin{minipage}{64mm}
\def\objectstyle{\scriptstyle}
\xymatrix@=6pt{
 1  & 4  &	3	&&\\
 2  &\ast \ar@{-}[dl] 		&&	\\
\ast			}
\end{minipage} 
\begin{minipage}{64mm}
\def\objectstyle{\scriptstyle}
\xymatrix@=6pt{
 3  & 2  &1	&	\\
 \ast  &\ast \ar@{-}[l] 		&	\\
 4			}
\end{minipage} \text{ are all $((3,2,1),1)$-tableaux.}
\intertext{\indent A \emph{row-standard} $(\lambda, i)$-\emph{tableau} is given by replacing the unconnected nodes with the numbers $\{1, \ldots, t-2i\}$, such that the numbers increase along the rows when read from left to right.  For example,}
&\begin{minipage}{64mm}
\def\objectstyle{\scriptstyle}
\xymatrix@=6pt{
1 & 2  &3	&	\\
 \ast  &\ast \ar@{-}[l] 		&	\\
 4			}
\end{minipage} \text{ is a row-standard $((3,2,1),1)$-tableaux.}
\end{align*}

 For a given row-standard diagram, $\sigma \in T_\lambda^i$, we define the \emph{restricted diagram}, $\lambda \!\!\downarrow_\sigma$, to be the composition of $t-2i$ obtained from $\lambda$ by deleting the connected nodes.  For example, there are exactly two row-standard $((3,1),1)$-diagrams
\begin{align*}
\sigma_1 = \begin{minipage}{64mm}
\def\objectstyle{\scriptstyle}
\xymatrix@=6pt{
 \ast  &  \ast  &	\ast		\\
 \ast \ar@{-}[u]   		}
\end{minipage} &\text{ and }
\sigma_2 = \begin{minipage}{64mm}
\def\objectstyle{\scriptstyle}
\xymatrix@=6pt{
 \ast  &  \ast \ar@{-}[r] &	\ast		\\
 \ast		}
\end{minipage}, 
\intertext{which correspond to the restricted diagrams}
\lambda\!\downarrow_{\sigma_1} \ = \ \begin{minipage}{34mm}
\def\objectstyle{\scriptstyle}
\xymatrix@=6pt{
 \ast  &	\ast	&		}
\end{minipage}   &\text{ and }
\lambda\!\downarrow_{\sigma_2} \ =\  \begin{minipage}{34mm} 
\def\objectstyle{\scriptstyle} 
\xymatrix@=6pt{
 \ast  \\	\ast			}
\end{minipage} \ .
\end{align*}
 
 The restricted tableaux are defined  similarly.  We define a dominance order on modified diagrams as follows. For $\sigma \in T^i_\lambda$ and $\tau \in T^j(\mu)$ we say that $\sigma \rhd \tau$ if and only if either: $|\lambda\!\downarrow_\sigma\!\!|<|\mu\!\downarrow_\tau\!\!|$, or $|\lambda\!\downarrow_\sigma\!\!|=|\mu\!\downarrow_\tau\!\!|$ and $ \mu\!\downarrow_\tau		\! \lhd_{t-2i} 		\lambda\!\downarrow_\sigma$.

For $\lambda \vdash t$ and $\omega \vdash t-2i$, we define a (semistandard) $\omega$-tableaux of type $(\lambda,i)$ to be a pair $\sigma \in T^i_\lambda$ and $\S \in T(\omega,\lambda\!\downarrow_\sigma)$ (respectively $\S \in T_0(\omega,\lambda\!\downarrow_\sigma)$), which we denote by $\S^\sigma$.   The set of all (semistandard) $\omega$-tableaux of type $(\lambda, i)$ is denoted $T^i(\omega, \lambda)$ (respectively $T_0^i(\omega, \lambda)$).  We let $T^\ast_0(\lambda)$ denote the union $\cup_iT_0^i(\omega, \lambda)$.
There is a unique $\lambda$-tableau of type $(\lambda,i)$ --- this is the tableau where $i=0$ and where the $i$th row consists of $\lambda_i$ copies of the integer $i$, we denote this tableau by $\T^\lambda$.

 \label{tabber}
 In the classical set-up, pairs of row-standard tableaux provide sets of (double) coset representatives of Young subgroups (through $\mathcal{D}_{\lambda,\mu}$).  The following lemma illustrates the connection between  modified tableaux and the double cosets in which we are interested.

\begin{lem}\label{ASETOFBRAUERCOSETS}
 
A set of double coset representatives for $\Sigma_{\lambda} \backslash \Sigma_t / \Sigma_{\mu} \times \Sigma_2\wr\Sigma_i$ is given by $\mathcal{D}_{\lambda\downarrow, \mu}^i = \{ \sigma \otimes \epsilon_{l+i} \otimes d_{\lambda\downarrow_\sigma, \mu} : \sigma \in T^i_\lambda, d_{\lambda\downarrow_\sigma, \mu} \in \mathcal{D}_{\lambda\downarrow_\sigma, \mu}\}$.

A set of coset representatives for $\Sigma_\lambda/(\Sigma_\lambda \cap(  \Sigma_{\mu} \times \Sigma_2\wr\Sigma_i))$ is given by 
$\mathcal{D}^i_\mu=\{\sigma \otimes \epsilon_{l+i} \otimes d_{\mu} : \sigma \in T^i_{\lambda}, d_{\mu} \in \mathcal{D}_\mu \cap \Sigma_{\lambda\downarrow_\sigma} \}.$   
\end{lem}
\begin{proof}

Arranging the elements of $T_{1^t}^i$ horizontally, we obtain all diagrams with $t$ nodes and $i$ edges, such that each dot borders at most one edge.  We have that $\Sigma_t$ acts  by permuting these diagrams and that $\Sigma_{t-2i} \times \Sigma_2\wr\Sigma_i$ is the stabiliser of the following diagram:
\begin{align*}
\begin{minipage}{64mm}
\def\objectstyle{\scriptstyle}
\xymatrix@=1pt{
&&\circ	&&\circ	&&\circ	&&\circ & &\ldots &&\circ&&\circ \ar@{-}[ll]		&&\circ	&&\circ \ar@{-}[ll]}
\end{minipage},
\end{align*}
where the $i$ arcs each join two consecutive nodes of the last $2i$ nodes.  Therefore, a set of left cosets representatives of $\Sigma_t/ \Sigma_{t-2i} \times \Sigma_2 \wr \Sigma_i$ is given by $T^i_{1^t}$.   

A set of double coset representatives for $\Sigma_{\lambda} \backslash \Sigma_t / \Sigma_{t-2i} \times \Sigma_2\wr\Sigma_i$ is equivalent to a parameterisation of the $\Sigma_{\lambda}$-orbits on $T^i_{1^t}$.   If we now rearrange the nodes  
 to form a $(\lambda,i)$-diagram we have that $\Sigma_{\lambda}$ acts by permuting the nodes in each row.  Therefore by definition,  we have that  $T^i_\lambda$ provides a set of double coset representatives for $\Sigma_\lambda \backslash \Sigma_t/ \Sigma_{t-2i} \times \Sigma_2\wr\Sigma_i$. 
The fiber of the projection from $\Sigma_{\lambda} \backslash \Sigma_t /  \Sigma_2\wr\Sigma_i$ to $\Sigma_{\lambda} \backslash \Sigma_t / \Sigma_{t-2i} \times \Sigma_2\wr\Sigma_i$ is $\mathcal{D}_{\lambda\downarrow_\sigma}$.  
The general case follows by taking $\Sigma_\mu$ cosets.

We may parameterise $\Sigma_\lambda/(\Sigma_\lambda \cap  \Sigma_{t-2i} \times \Sigma_2\wr\Sigma_i)$ cosets in a similar fashion.  Here, the stabiliser of a diagram, $\sigma$, is given by $\Stab(\sigma)=\Sigma_\lambda \cap \Sigma_{t-2i} \times \Sigma_2\wr\Sigma_i$.  These orbits are readily seen to be parameterised by $T^i_\lambda$.  
 \end{proof}

\subsection{Filtrations of permutation modules}\label{brauer proof} 
 The following is a restatement, in the language of modified tableaux, of several of the results from \cite[Section 7]{Row}.  There they show that $e_{l+i}M(\lambda)\simeq \oplus_{\sigma\in T^i_{\lambda}}M(\lambda\!\downarrow_\sigma)$ and give the isomorphism below.
%

\begin{prop}\label{subqBrauer}
   The permutation modules, $M(\lambda,l)$, for the Brauer algebra have a filtration with subquotients of the form $\oplus_{\sigma \in T^i_\lambda} 
    \G_{l+i}(M(\lambda\!\downarrow_\sigma)).$  
   The isomorphism is as follows:
 \begin{align*}
\pi:
 \G_{l+i}(M(\lambda\!\downarrow_\sigma))
&\longrightarrow 
J_{l,i}/J_{l,i+1}
\otimes_{K\Sigma_{t}}  M(\lambda)
 \\
 v \otimes \epsilon_{l+i} \otimes x_{\lambda\downarrow_\sigma} &\hookrightarrow (v \otimes \sigma \otimes \id) \otimes_{K\Sigma_t} x_{\lambda}.
 \end{align*}
\end{prop}

This proposition allows us to immediately conclude that semistandard modified tableaux give a filtration of permutation modules by standard modules.

\begin{cor}
The $B_K(r,\delta)$-module $M(\lambda,l)$ has a filtration $M(\lambda,l)=M_1   \supseteq \ldots \supseteq M_{k+1} =0$ such that there exists $\mu^j\vdash t-2i$ with $M_j/M_{j+1} \simeq \Delta(\mu^j,l+i)$. Moreover, for each partition $\mu$ the number of $\mu^j$ equal to $\mu$ is the number of semistandard $\mu$-tableaux of type $(\lambda,i)$.
\end{cor}

\begin{eg}
Consider the $B_K(4,\delta)$-module $M((3,1),0)$.   There are two row-standard $((3,1),1)$-diagrams, namely
\begin{align*}
\sigma_1 = \begin{minipage}{54mm}
\def\objectstyle{\scriptstyle}
\xymatrix@=6pt{
\ast		&\ast &\ast	\ar@{-}[dll] &	 \\
\ast 
}
\end{minipage}  
\text{and } \ \ 
\sigma_2=\begin{minipage}{54mm}
\def\objectstyle{\scriptstyle}
\xymatrix@=6pt{
\ast 	&\ast \ar@{-}[r]		&\ast \\
\ast 
}
\end{minipage}.
 \end{align*}
 We have that $(3,1)\!\!\downarrow_{\sigma_1}=(2)$ and $(3,1)\!\!\downarrow_{\sigma_2}=(1^2)$, therefore
\begin{align*}
J_{0,1}/J_{1,1} \otimes_{K\Sigma_4}M(3,1) \simeq \G_1(M(2)) \oplus \G_1(M(1^2)).
\end{align*} 
\end{eg}

\subsection{A Murphy basis}\label{Smurphy} We now give a basis of the permutation modules which generalises the construction of Murphy.
 \begin{prop}[A Murphy basis]
The permutation module $M(\lambda,l)$ has a basis 
 given in terms of modified tableaux as follows:
\begin{align*}
\{ v \otimes [\sigma] \otimes m_{ \mathfrak{s}\T^\sigma} :  v \in {V_{l+i}}, \T^\sigma \in T^i_0(\omega, \lambda), \mathfrak{s} \in \Std(\omega) \text{ for some } \omega \vdash t \}.
\end{align*}
\end{prop}

\begin{remk}
The permutation module $M(1^r,0)$ is isomorphic to $B_K(r, \delta)$.  In this case the above basis is a well known cellular basis of the Brauer algebra with respect to the dominance ordering. 
\end{remk}

 \begin{proof}[Proof of Proposition]
 By Proposition \ref{subqBrauer}, it is enough to fix $\sigma \in T^i_\lambda$,  and consider the basis of the subquotient $\{(v \otimes \sigma \otimes \id) \otimes_{K\Sigma_t} M(\lambda)$: $v \in V_{l+i}\}$.  Fix $v \in V_{l+i}$, we can move $x_\lambda$ through the tensor product, so that the bottom rows of diagrams in $Ae_l$ are given by $[\sigma]$ and the  through-lines are given by $d_\sigma x_{\lambda\downarrow_\sigma}$, for $d_\sigma \in \mathcal{D}_{\lambda\downarrow_\sigma}$.  We may take the basis of Theorem \ref{Smurphet} to provide the through lines and so we are done.  
 %
 \end{proof}
\begin{eg}
We consider the $B_K(4,\delta)$-module $M((3,1),0)$.  We want to construct a basis of the subquotient $J_{0,1}/J_{0,2} \otimes_{\Sigma_4}M(3,1)$.  
 We have that 
\begin{align*}
\epsilon_1 \otimes [\sigma_1] \otimes x_{\lambda\downarrow_{\sigma_1}} = \ \
 &\begin{minipage}{34mm}
\def\objectstyle{\scriptstyle}
\xymatrix@=1pt{
\circ 	\ar@{-}[dd]	 &&\circ 	\ar@{-}[dd]			&&\circ\ar@{-}[rr]			&&\circ	  \\	&\\
\circ  	&&\circ				&&\circ \ar@{-}[rr]		&&\circ								 } \end{minipage} \ + \
\begin{minipage}{34mm}
\def\objectstyle{\scriptstyle}
\xymatrix@=1pt{
\circ 	\ar@{-}[dd]	 &&\circ 	\ar@{-}[ddrr]			&&\circ\ar@{-}[rr]			&&\circ \\ &\\
\circ  				&&\circ	 \ar@{-}@/^.7pc/[rrrr]		 		&&\circ	&&\circ} \end{minipage} \ + \
 \begin{minipage}{34mm}
\def\objectstyle{\scriptstyle}
\xymatrix@=1pt{
\circ 	\ar@{-}[ddrr]	 &&\circ 	\ar@{-}[drdr]			&&\circ\ar@{-}[rr]			&&\circ \\ &\\
\circ  \ar@{-}@/^.7pc/[rrrrrr] 	&&\circ				&&\circ 	&&\circ} \end{minipage}\\
\ + \ &\begin{minipage}{34mm}
\def\objectstyle{\scriptstyle}
\xymatrix@=1pt{
\circ 	\ar@{-}[ddrr]	 &&\circ 	\ar@{-}[ddll]			&&\circ\ar@{-}[rr]			&&\circ \\&\\
\circ  	&&\circ				&&\circ \ar@{-}[rr]		&&\circ } \end{minipage} \ + \
\begin{minipage}{34mm}
\def\objectstyle{\scriptstyle}
\xymatrix@=1pt{
\circ 	\ar@{-}[ddrrrr]	 &&\circ 	\ar@{-}[ddll]			&&\circ\ar@{-}[rr]			&&\circ\\ &\\
\circ  	&&\circ	 \ar@{-}@/^.7pc/[rrrr]		 		&&\circ	&&\circ} \end{minipage} \ + \
 \begin{minipage}{34mm}
\def\objectstyle{\scriptstyle}
\xymatrix@=1pt{
\circ 	\ar@{-}[ddrrrr]	 &&\circ 	\ar@{-}[dd]			&&\circ\ar@{-}[rr]			&&\circ\\ &\\
\circ  \ar@{-}@/^.7pc/[rrrrrr] 	&&\circ				&&\circ 	&&\circ} \end{minipage},\\
\intertext{and}
\epsilon_1 \otimes [\sigma_2] \otimes x_{\lambda\downarrow_{\sigma_2}} = \ \ %
& \begin{minipage}{34mm}
\def\objectstyle{\scriptstyle}
\xymatrix@=1pt{
\circ 	\ar@{-}[dd]	 &&\circ 	\ar@{-}[ddrrrr]			&&\circ\ar@{-}[rr]			&&\circ  \\ &\\
\circ  	&&\circ				&&\circ \ar@{-}[ll]		&&\circ } \end{minipage} \ + \
\begin{minipage}{34mm}
\def\objectstyle{\scriptstyle}
\xymatrix@=1pt{
\circ 	\ar@{-}[ddrr]	 &&\circ 	\ar@{-}[ddrrrr]			&&\circ\ar@{-}[rr]			&&\circ  \\ &\\
\circ	 \ar@{-}@/^.7pc/[rrrr]		 		&&\circ &&\circ	&&\circ } \end{minipage} \ + \
 \begin{minipage}{34mm}
\def\objectstyle{\scriptstyle}
\xymatrix@=1pt{
\circ 	\ar@{-}[ddrrrr]	 &&\circ 	\ar@{-}[ddrrrr]			&&\circ\ar@{-}[rr]			&&\circ  \\ &\\
\circ  \ar@{-}[rr] 	&&\circ				&&\circ 	&&\circ} \end{minipage}.
\end{align*}
 The subquotient $J_{0,1}/J_{0,2}\otimes_{K\Sigma_4} M(3,1)$ splits into a direct sum of two modules. 
 The submodule isomorphic to ${\G_1} (M(2))$ of the subquotient has basis 
\begin{align*} \{	v  \otimes (
 (\begin{minipage}{34mm}
\def\objectstyle{\scriptstyle}
\xymatrix@=4pt{ \circ  	&\circ				&\circ \ar@{-}[r]		&\circ	} \end{minipage})+
(\begin{minipage}{34mm}
\def\objectstyle{\scriptstyle}
\xymatrix@=4pt{ \circ  	& 	 \circ\ar@{-}@/_.4pc/[rr]		&\circ 	&\circ	} \end{minipage})
+
(\begin{minipage}{34mm}
\def\objectstyle{\scriptstyle}
\xymatrix@=4pt{ \circ\ar@{-}@/_.4pc/[rrr]	 	& \circ 		&\circ 		&\circ	} \end{minipage}))
 \otimes (1 + (1 \ 2)) : v \in {V_1} \},
\intertext{and the submodule isomorphic to ${\G_1}(M(1^2))$ has basis} 
 \{	v  \otimes (
(\begin{minipage}{34mm}
\def\objectstyle{\scriptstyle}
\xymatrix@=4pt{ \circ  	&\circ				&\circ \ar@{-}[l]		&\circ	} \end{minipage})+
(\begin{minipage}{34mm}
\def\objectstyle{\scriptstyle}
\xymatrix@=4pt{ \circ\ar@{-}@/_.4pc/[rr]	  	&\circ		&\circ 	&\circ	} \end{minipage})
+
(\begin{minipage}{34mm}
\def\objectstyle{\scriptstyle}
\xymatrix@=4pt{ \circ  	&\circ	  \ar@{-}[l]				&\circ 		&\circ	} \end{minipage}))
 \otimes  g : v \in {V_1}, g \in \Sigma_2 \}.
\end{align*}
\end{eg}

\subsection{Homomorphisms between permutation modules}\label{frob1}
 Using the adjunction in Section \ref{ind} and the realisation of the globalisation functor in Proposition \ref{EXAMbrauer}, we describe homomorphisms between permutation modules.
  
\begin{prop}\label{piranese} 
Let $\lambda \vdash t$, $\mu \vdash s$, then $\{\varphi^d_{\sigma, \tau} : \sigma \in T_\lambda, \tau \in T_\mu, d \in \mathcal{D}_{\lambda\downarrow_\sigma \mu\downarrow_\tau  }\}$ is a basis of $\Hom_{B_r}(M(\lambda,l), M(\mu,m))$, where $\varphi^d_{\sigma, \tau}$ is given by
\begin{align*}
\varphi_{\sigma\tau}^d(\epsilon_l\otimes\epsilon_l\otimes x_\lambda) &= [\sigma]	\otimes[\tau]	\otimes \sum_{g \in \mathcal{D}_\nu \cap \Sigma_{\lambda\downarrow_\sigma}} gd x_\mu , 
\end{align*}
where $\nu$ is the composition which labels the Young subgroup $d\Sigma_{\mu\downarrow_\tau}d^{-1}\cap\Sigma_{\lambda\downarrow_\sigma}\leq\Sigma_{t-2i}$.  
\end{prop}

\begin{rmk}
In the above we have used Theorem \ref{DJbasisofschur} to determine the through-lines; this is to allow direct comparison between the above and \cite[Theorem 5.3]{hko}.  Note that later we shall instead use Theorem \ref{cellbasisofschur} as this shall result in a cellular basis.
\end{rmk}

\begin{proof} Using the adjunction in Section \ref{ind} we have that
\begin{align*} 
\Hom_{B_r}(M(\lambda,l), M(\mu,m)) &\simeq
\Hom_{\Sigma_t}(M(\lambda),  e_l(\oplus_i (J_{m,i}/J_{m,i+1})\otimes_{K\Sigma_{r-2m}} M(\mu))) \tag{$\star$} \\
\intertext{where the left action of $K\Sigma_t$ is by permutation of arcs (it cannot increase the number of arcs), hence the direct sum.  By Proposition \ref{subqBrauer} we have that}
&\simeq
\Hom_{\Sigma_t}(M(\lambda), e_l(\oplus_{\tau\in T^i_\mu} {^0V^{m+i}}\otimes_{K\Sigma_{r-2(m+i)}}  M(\mu\!\downarrow_\tau))), 
\intertext{if $l>m$, multiplication by $e_l$ will annihilate the top $(l-m)$ layers of $M(\mu,m)$ (by Lemma \ref{restrict}).  We set $j$ to be $l-m+i$ if $l-m\geq0$ and to be $i$ otherwise.  Now,			}
&\simeq\Hom_{\Sigma_t}(M(\lambda), \oplus_{\tau\in T^j_\mu} {^lV^{l+i}}\otimes_{K\Sigma_{r-2(l+i)}} M(\mu\!\downarrow_\tau)) \\
&\simeq \Hom_{\Sigma_t}( K_{\Sigma_\lambda}\!\!\uparrow^{\Sigma_t},  \oplus_{\tau\in T^j_\mu} K_{ \Sigma_{\mu\!\downarrow_\tau}\times\Sigma_2 \wr \Sigma_i}\!\!\uparrow^{\Sigma_{t}}).  \tag{$\star\star$}
\end{align*}
We now fix a direct summand in the above by picking an element $\tau \in T^j_\mu$. By Frobenius reciprocity we have that
\begin{align*}
\Hom_{\Sigma_t}(K_{\Sigma_\lambda}\!\uparrow^{\Sigma_t},K_{\Sigma_{\mu\downarrow_\tau} \times \Sigma_2 \wr \Sigma_i}\!\uparrow^{\Sigma_t}) 
&\simeq \Hom_{\Sigma_\lambda}(K , K_{\Sigma_{\mu\downarrow_\tau} \times \Sigma_2 \wr \Sigma_i}\!\uparrow^{\Sigma_t}\downarrow_{\Sigma_\lambda}).
\intertext{Using our set of double cosets and Mackey decomposition we have:}
\Hom_{\Sigma_t}(K_{\Sigma_\lambda}\!\uparrow^{\Sigma_t},K_{\Sigma_{\mu\downarrow_\tau} \times \Sigma_2 \wr \Sigma_i}\!\uparrow^{\Sigma_t}) &\simeq 	\oplus_{\sigma \in T^i_\lambda}		\Hom_{\Sigma_\lambda}(K , K\!\downarrow_{\Sigma_\lambda \cap (\Sigma_{\mu\downarrow_\tau} \times \Sigma_2 \wr \Sigma_i)^\sigma }\uparrow^{\Sigma_\lambda}),
\end{align*}
where the superscript $\sigma$ denotes that this is the embedding of the subgroup into $\Stab(\sigma)$.
We sum over the elements of this orbit to obtain the corresponding homomorphism 
\begin{align*}
\varphi_\sigma : &K_{\Sigma_\lambda}\!\uparrow^{\Sigma_t} \to K_{\Sigma_{\mu\downarrow_\tau} \times \Sigma_2 \wr \Sigma_i}\!\uparrow^{\Sigma_t} \\
& x_\lambda \longmapsto [\sigma] \otimes \epsilon_{l+i} \otimes x_{t-2i}.
\end{align*}
This is an element of $(\star\star)$, by Propositions \ref{subqBrauer} and \ref{Smurphy} the isomorphic element in $(\star)$ is obtained by substituting $[\tau]$ in the place of $\epsilon_{l+i}$.
\end{proof}

\begin{eg}
Consider our earlier example of $\Hom_{B_K(3,\delta)}(M((3),0),M((1),1)).$ We have that there is a single 
$\sigma = 
(\def\objectstyle{\scriptstyle}
\xymatrix@=6pt{
 \ast  &	\ast  \ar@{-}[l]	&\ast
}) \in T^1_3$, a single 
$\tau = ( \def\objectstyle{\scriptstyle}
\xymatrix@=6pt{
 \ast }) \in T^0_1$,
 and a single  $1_{\Sigma_3}  \in \mathcal{D}_{3\downarrow_\sigma, 1}$.  Therefore, there will be a unique homomorphism, $\varphi_{\sigma, \tau}^1$, determined by
\begin{align*}
\varphi_{\sigma, \tau}^1(x_{(3)}) =[\sigma] \otimes [\tau] \otimes K&=
  \begin{minipage}{34mm}
\def\objectstyle{\scriptstyle}
\xymatrix@=5pt{
\circ 	\ar@{-}[d]	&\circ \ar@{-}[r]	&\circ 		\\
\circ		&\circ 	\ar@{-}[r]	&\circ 	
}
\end{minipage}+
\begin{minipage}{34mm}
\def\objectstyle{\scriptstyle}
\xymatrix@=5pt{
\circ 	\ar@{-}[r]	&\circ 	&\circ  	\ar@{-}[dll] \\
\circ		&\circ 	\ar@{-}[r]	&\circ 	}
\end{minipage}+
  \begin{minipage}{34mm}
\def\objectstyle{\scriptstyle}
\xymatrix@=5pt{
\circ\ar@{-}@/_.8pc/[rr] 		&\circ \ar@{-}[dl]	&\circ	\\
\circ		&\circ 	\ar@{-}[r]	&\circ 	
}
\end{minipage}.
\end{align*}
We now specialise to the case $\delta=-2$.  
 We have that the kernel of $\varphi_{\sigma, \tau}^1$  is the 3-dimensional submodule 
$J_{0,1} \otimes_{K\Sigma_3} M((3),0) \subset M((3),0)$,   
as can be seen by the following:
\begin{align*}
\varphi_{\sigma, \tau}^1 (e_1 (\epsilon_0 \otimes \epsilon_0 \otimes x_{(3)}))
&= e_1  \varphi_{\sigma, \tau}^1 (  
\epsilon_0 \otimes \epsilon_0 \otimes  x_{(3)}) 		\\
&=	e_1( \begin{minipage}{34mm}
\def\objectstyle{\scriptstyle}
\xymatrix@=5pt{
\circ 	\ar@{-}[d]	&\circ \ar@{-}[r]	&\circ 		\\
\circ		&\circ 	\ar@{-}[r]	&\circ 	
}
\end{minipage}+
\begin{minipage}{34mm}
\def\objectstyle{\scriptstyle}
\xymatrix@=5pt{
\circ 	\ar@{-}[r]	&\circ 	&\circ  	\ar@{-}[dll] \\
\circ		&\circ 	\ar@{-}[r]	&\circ 	}
\end{minipage}+
  \begin{minipage}{54mm}
\def\objectstyle{\scriptstyle}
\xymatrix@=6pt{
\circ\ar@{-}@/_.8pc/[rr] 		&\circ \ar@{-}[dl]	&\circ	\\
\circ		&\circ 	\ar@{-}[r]	&\circ 	
}
\end{minipage})\\
&=(-2+1+1)e_1 =0.
\end{align*}
Note that reflection through the horizontal axis gives $\varphi_{\tau,\sigma}^1\in\Hom_{B_K(3,\delta)}(M((1),1),M((3),0)).$
\end{eg}

\section{The walled Brauer algebra}

In this section we shall provide the combinatorial set-up for the walled Brauer algebra.  Many of the proofs are identical to the Brauer algebra case, and so are omitted.
 
\subsection{Definitions and examples}\label{quantum}

It is easy to define the walled Brauer algebra, $BW_K(r', r,\delta)$, (sometimes denoted $BW_{r',r}$) as a subalgebra of the Brauer algebra $B_K(r'+r,\delta)$.  We partition a basis diagram with a wall separating the first $r'$ northern and southern nodes from the remainder.  Then the walled Brauer algebra is the subalgebra with basis given by the diagrams such that all arcs cross the wall, and no through-lines do so.  For example
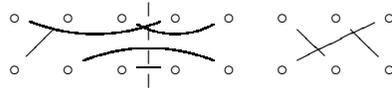
\begin{align*}
\begin{minipage}{54mm}
\def\objectstyle{\scriptstyle}
\xymatrix@=1pt{
&&						&&				&&		&\ar@{--}[dddd]\\
&&\circ \ar@{-}@/_.6pc/[rrrrrr]	&&\circ  			&&\circ 	& 				 &\circ 			&&\circ \ar@{-}@/^.6pc/[llll]		&&\circ  \ar@{-}[ddrr]		&&\circ 					&&\circ\ar@{-}[ddllll]	  \\
&&\\	
&&\circ \ar@{-}[uurr] 			&&\circ				&& \circ  \ar@{-}[rr]		&				&\circ 			&&\circ \ar@{-}@/_.6pc/[llllll]				&&\circ 						&&\circ 			&&\circ \ar@{-}[uull]\\
&&						&&				&&			&}
\end{minipage} 
\end{align*}
is an element of $BW_K(3,5,\delta)$.

 \begin{prop}
Let  $r',r$ be integers, and $\delta \in K$.  If $r', r =1$ then suppose $\delta \neq 0$. The  walled Brauer algebra, $BW_K({r',r},\delta)$, is cellularly stratified.
\end{prop}
\begin{proof}
In \cite{coxwall} it is shown that the walled Brauer algebra is an iterated inflation with input algebras  $\Sigma_{r'-l}\times\Sigma_{r-l}$ for $l \leq {\rm min}\{r',r\}$.  We let $\Sigma_{r',r}$ denote $\Sigma_{r'}\times\Sigma_{r}$.   They define idempotents $e_l$ for $l \leq \minl \{r,r'\}$, as follows.  If $\delta \neq 0$ then we define $e_l$ to be $1/\delta^l$ times the basis element
\begin{align*}
\begin{minipage}{54mm}
\def\objectstyle{\scriptstyle}
\xymatrix@=1pt{
&&						&&				&&		&\ar@{--}[dddd]\\
&&\circ	&\ldots&  \circ \ar@{-}@/_.6pc/[rrrrrr]			&&\circ 				& & \circ \ar@{-}[ll]			&&\circ 	&\ldots&\circ &&\circ 			\\
&&\\	
&&\circ \ar@{-}[uu] 			&\ldots&\circ			&& \circ  \ar@{-}[rr]		&&\circ 			&&\circ \ar@{-}@/_.6pc/[llllll]				&\ldots&	\circ \ar@{-}[uu]	&&	\circ \ar@{-}[uu]	\\
&&						&&				&&			&}
\end{minipage} ,
\end{align*}
with $l$ northern and southern arcs connecting the nodes $r'-j$ to $r'+j$ for $j \leq l$, and $r-2l$ straight through-lines connecting the remaining nodes. 

 These idempotents are readily seen to satisfy the necessary conditions given in Definition \ref{cellDEF}.
 If $\delta=0$ and $r'$ or $r$ is at least 2 
 then one can define alternative idempotents which satisfy the necessary conditions (see \cite{coxwall}).
 \end{proof}

\begin{prop}\label{walledcaseofinflation}
  Let $(r',r)=(t+l,t+l)=(s'+m, s+m)$, let $m-l=i\in \mathbb{Z}$, and assume without loss of generality that $i\geq0$.    Let ${^mV^l}$ denote the $K\Sigma_{s',s} - K\Sigma_{t',t} $-bimodule  $e_m(A/Ae_{l+1}A)e_l$.
We have that 
  \begin{align*}
{^lV^m}\otimes_{K\Sigma_{s',s}} - : \Sigma_{s'}\times \Sigma_{s}{\text{\rm -mod}} \xrightarrow{\rm inflation} & \Sigma_{s'} \times  \Sigma_i  \times \Sigma_{s} {\text{\rm -mod}} \xrightarrow{\rm induction}\Sigma_{t'}\times\Sigma_{t}{\text{\rm -mod}}.
  \end{align*}
Similarly, we have that
\begin{align*}
{^mV^l}\otimes_{K\Sigma_{t',t}} - : \Sigma_{t'}\times\Sigma_{t}{\text{\rm -mod}}& \xrightarrow{\rm restriction}  \Sigma_{s'} \times  \Sigma_i  \times \Sigma_{s} {\text{\rm -mod}} \xrightarrow{\rm projection}\Sigma_{s'}\times \Sigma_{s}{\text{\rm -mod}} 
 \end{align*}
and we shall let $e_{l+i}N$ denote the image of $N$ under the localisation map.  
Furthermore, $$J_{l,i}/J_{l,i+1}\otimes_{K\Sigma_{t',t}} M \simeq J_{l+i,0}/J_{l+i,1}\otimes_{K\Sigma_{t'-i,t-i}}(e_{l+i}M).$$ 
\end{prop}
\begin{proof}
In the case of the Brauer algebra, the globalisation module was a transitive permutation module, with the stabiliser of a set of arcs given by $\Sigma_2\wr\Sigma_i\leq\Sigma_{2i}$.  In the case of the walled Brauer algebra, the globalisation module is given by the transitive permutation module for the group $\Sigma_i \times \Sigma_i$, with basis consisting of diagrams with arcs which cross the wall. The stabiliser of a given element, say
\begin{align*}
\begin{minipage}{54mm}
\def\objectstyle{\scriptstyle}
\xymatrix@=1pt{
&&						&&				&&		&\ar@{--}[dd]\\
&&\circ	&\ldots&  \circ \ar@{-}@/_.6pc/[rrrrrr]			&&\circ 				& & \circ \ar@{-}[ll]			&&\circ 	&\ldots&\circ &&\circ 		\\	&&						&&				&&			&}
\end{minipage} ,
\end{align*}
  is given by the diagonal embedding $\Sigma_i < \Sigma_i \times \Sigma_i$. With the diagrams in place, we may now argue as in the case of the classical Brauer algebra.
\end{proof}

\subsection{Modified tableaux}

In the case of the walled Brauer algebra we shall always take $(r',r)= (t'+l, t+l)= (s'+m, s+m)$.  
We shall take $i = m-l \in \mathbb{Z}$. 
 In this section we shall only need to consider the case of $m-l \geq 0$, as before.

\subsubsection{Definitions and examples}A bicomposition  $\lambda=(\lambda^a,\lambda^b)  \vdash (t', t)$  can be illustrated by a generalised Young diagram, $[\lambda^a,\lambda^b]$, as illustrated below,
\begin{align*}
&\begin{minipage}{54mm}
\def\objectstyle{\scriptstyle}
\xymatrix@=6pt{
&&\ast	&\ast 	\\
 \ast  &  \ast  &	\ast	&\ast 	\\
&&&& \ast 	&  \ast  &	\ast	 	\\
&&&& \ast   &\ast 			\\
 	 		}
\end{minipage} 
\  \text{ is a $((4,2),(3,2))$ Young diagram.}
\end{align*}

 A $(\lambda, i)$-\emph{diagram} arises from a pair $(\lambda, i)$ by taking the Young diagram $[\lambda^a,\lambda^b]$ and connecting $i$ nodes from the upper left half of the diagram to the lower right half of the diagram, and neglecting the order of points in the row.  

Define an equivalence relation on modified diagrams by neglecting the order of the points in the row.  The  \emph{row-standard modified diagrams} $\sigma \in T^i_\lambda$, cosets, and $[\sigma]$ are defined as before.
 For example
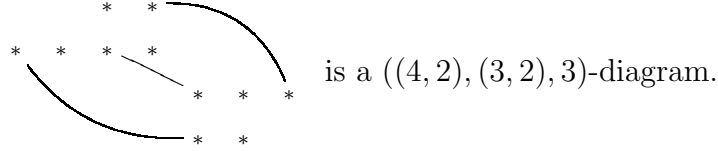
\begin{align*}
&\begin{minipage}{54mm}
\def\objectstyle{\scriptstyle}
\xymatrix@=6pt{
&&\ast	&\ast 	\\
 \ast  &  \ast  &	\ast	&\ast 	\\
&&&& \ast\ar@{-}[ull]  	&  \ast  &	\ast	\ar@{-}@/_1pc/[uulll]	\\
&&&& \ar@{-}@/^1pc/[uullll]\ast   &\ast 			\\
		}
\end{minipage} 
\ \text{ is a $((4,2), (3,2),3)$-diagram.}
\end{align*}

A $(\lambda, i)$-\emph{tableau} arises from a row-standard $(\lambda, i)$-diagram by replacing the first $t'-i$ unconnected nodes with the integers $\{1, \ldots, t'-i\}$ in some order, and the final $t-i$ unconnected nodes with the integers $\{t' +1+i, \ldots, t'+t\}$.  For example,
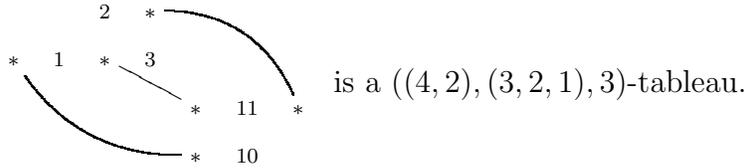
\begin{align*}
&\begin{minipage}{44mm}
\def\objectstyle{\scriptstyle}
\xymatrix@=6pt{
&&2	&\ast 	\\
 \ast  &  1 &	\ast	&3	\\
&&&& \ast\ar@{-}[ull]  	&  11 &	\ast	\ar@{-}@/_1pc/[uulll]	\\
&&&& \ast  \ar@{-}@/^1pc/[uullll] &10	\\
}
\end{minipage} 
\ \text{ is a $((4,2), (3,2,1),3)$-tableau.} 
\end{align*}

A \emph{row-standard} $(\lambda, i)$-\emph{tableau} is given by replacing the unconnected nodes with  numbers that increase along the rows when read from left to right.

 For a given row-standard diagram, $\sigma \in T_\lambda^i$, we define the \emph{restricted diagram}, $\lambda\! \!\downarrow_\sigma$, to be the bicomposition of $(t'-i,t-i)$ obtained from $\lambda$ by deleting the connected nodes.
The dominance order, construction of semistandard tableaux etc. are all defined analogously to the Brauer algebra case.   
  
 \subsubsection{Diagrammatic coset representatives} \label{flibjib}
 The following lemma illustrates the connection between modified tableaux and the double cosets in which we are interested.

 \begin{lem}
A set of double coset representatives for $\Sigma_{\lambda^a,\lambda^b} \backslash \Sigma_{t',t} / \Sigma_{\mu^a,i,\mu^b}$ is given by $\mathcal{D}_{\lambda\downarrow, \mu}^i = \{ \sigma \otimes \epsilon_{l+i} \otimes d_{\lambda\downarrow_\sigma, \mu} : \sigma \in T^i_\lambda, d_{\lambda\downarrow_\sigma, \mu} \in \mathcal{D}_{\lambda\downarrow_\sigma, \mu}\}$.

A set of coset representatives for $\Sigma_{\lambda^a,\lambda^b} / (\Sigma_{\lambda^a,\lambda^b}\cap \Sigma_{\mu^a,i,\mu^b})$ is given by $\mathcal{D}^i_\mu=\{\sigma \otimes \epsilon_{l+i} \otimes d_{\mu} : \sigma \in T^i_{1^t}, d_{\mu} \in \mathcal{D}_\mu\cap \Sigma_{\lambda \downarrow_\sigma} \}.$
\end{lem}
 
 \begin{proof}
Arranging the elements of $T^i_{1^t}$ horizontally we get obtain all diagrams with $t'+t$ dots, a wall dividing the first $t'$ dots from the final $t$ dots, and $i$ edges 
crossing the wall.  We have that $\Sigma_{t',t}$ acts by transitively permuting these diagrams and that $\Sigma_{t'-i} \times \Sigma_i \times \Sigma_{t-i}$ is the stabiliser of the following diagram:
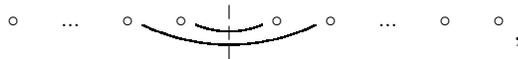
\begin{align*}
\begin{minipage}{64mm}
\def\objectstyle{\scriptstyle}
\xymatrix@=1pt{
&&&&		&&		&&		&&\ar@{--}[ddd]	\\
&&\circ &&\ldots	&&\circ	&&\circ	&&			&&\circ	\ar@{-}@/^.43pc/[llll] &&\circ\ar@{-}@/^.85pc/[llllllll] &&\ldots	&&\circ 	&&\circ	  \\ &\\
&&	&&		&&		&&		&&			&&&&}
\end{minipage},
\end{align*}
where the $l+i$ arcs join the nodes $t'+1-j$ and $t'+j$ for $j \leq l+i$.  We remark that the stabiliser of this element comes from the diagonal embedding of $\Sigma_i \hookrightarrow \Sigma_i \times \Sigma_i$.  

With the diagrams in place, one can now adapt the proof of Proposition \ref{ASETOFBRAUERCOSETS}.
\end{proof}

\subsection{Filtrations of permutation modules} 
We now show that the permutation modules have a filtration similar to that of \cite[Section 7]{Row}.  
\begin{prop}\label{flibbydejibjob}
The permutation modules, $M(\lambda,l)$, for the algebra have a filtration with subquotients of the form $\oplus_{\sigma \in T^i_\lambda}\G_{l+i} (M(\lambda\!\downarrow_\sigma)) \simeq J_{l,i}/J_{l,i+1} \otimes_{K\Sigma_t}M(\lambda)$.  
\end{prop}
\begin{proof}
By Proposition \ref{walledcaseofinflation} we  need to show that $e_{l+i}M(\lambda) = M(\lambda) /\langle m-hm : h \in \Sigma_i \rangle$ is isomorphic to $\oplus_{\sigma \in T^i_\lambda} M (\lambda\!\!\downarrow_\sigma)$.
We can choose a basis of $e_{l+i}M(\lambda)$ to be a subset of the permutation basis of $M(\lambda)$ by choosing a representative of each $\Sigma_{i}$-orbit. 

It is clear that the induced action of $\Sigma_{t'-i,t-i}$ is again given by permutation of these elements.  We have that the stabiliser of an element $m_j \boxtimes m_k$ is given by the intersection of $\Sigma_{t'-i,t-i}$ and the stabiliser of $m_j \boxtimes m_k$ in $\Sigma_{t',t}$.  As both are (products of) Young subgroups, their intersection is one too.
Finally, note that the permutation modules which occur are identified by $\Sigma_\lambda$-orbits, and therefore
$e_{l+i}M(\lambda) \simeq \oplus_{\sigma \in T^i_\lambda} M (\lambda\!\!\downarrow_\sigma)$.
\end{proof}

\begin{cor}[`Specht' series of permutation modules]
The $BW_K({r',r},\delta)$-module $M(\lambda,l)$ has a filtration $M(\lambda,l)=M_1 \supseteq M_2 \supseteq \ldots \supseteq M_{k+1}=0$ such that there exists $\mu^j\vdash t-2i$ with $M_j/M_{j+1} \simeq \Delta(\mu^j,l+i)$. Moreover, for each partition $\mu$ the number of $\mu^j$ equal to $\mu$ is the number of semistandard $\mu$-tableaux of type $(\lambda,i)$.
\end{cor}

\subsection{A Murphy basis} By proceeding as in the Brauer algebra case, we get a Murphy basis for the permutation modules of the walled Brauer algebra. 
\begin{prop}[A Murphy basis]
The permutation module $M(\lambda,l)$ has a basis
in terms of modified tableaux as follows:
\begin{align*}
\{ v \otimes [\sigma] \otimes m_{ \mathfrak{s}\T^\sigma} :v \in {V_{l+i}}, \T^\sigma \in T^i_0(\omega, \lambda), \mathfrak{s} \in \Std(\omega) \text{ for some } \omega \vdash t\}.
\end{align*}
\end{prop}

\subsection{Homomorphisms between permutation modules}\label{frob2}
 Using the adjunction in Section \ref{ind} and the realisation of the globalisation functor in Proposition \ref{quantum}, we describe homomorphisms between permutation modules.

\begin{prop}\label{porg}
Let $\lambda \vdash (t',t)$, $\mu \vdash (s',s)$, then $\{\varphi_{\S^\sigma \T^\tau} : \S^\sigma \in T_0^\ast(\omega,\lambda), \T^\tau \in T_0^\ast(\omega,\mu)\}$ is a basis of $\Hom_{BW_{r',r}}(M(\lambda,l), M(\mu,m))$, where $\varphi_{\S^\sigma\T^\tau}$ is given by
\begin{align*}
\varphi_{\S^\sigma \T^\tau} (\epsilon_l\otimes\epsilon_l\otimes x_\lambda) &= [\sigma]	\otimes[\tau]	\otimes m_{{\S  \T}}.  
\end{align*}
\end{prop}
\begin{proof}
This is the same as the Brauer algebra case, except that we determine the through lines using the basis of Theorem \ref{cellbasisofschur}, rather than that of Theorem \ref{DJbasisofschur}.
\end{proof}

\section{The partition algebra}
The partition algebra is more complicated than the Brauer and walled Brauer algebras.  In particular the globalisation functor is more complicated; this leads to the need to define modified tableaux and bi-tableaux.  
These bi-tableaux are needed for the proofs of the main results (Murphy basis theorem etc), but not for their statements.  
\subsection{Definitions and examples}The partition algebra, $P_K(r,\delta)$, has a basis given by diagrams of $r$ northern and $r$ southern nodes, with edges connecting arbitrarily many of  these nodes.  The Brauer algebra can be seen as the subalgebra in which all edges connect precisely two nodes.  The multiplication is defined by concatenation similarly to the Brauer algebra case. 

In \cite{HHKP} the partition algebra with parameter $\delta\neq0$ was shown to be cellularly stratified with inflation decomposition
  \begin{align*}
P_K(r,\delta) = \oplus_{l} V_l \otimes V_l \otimes \Sigma_{r-l}.
\end{align*}
where $V_l$ is the vector space spanned by dangles of the following form.

\begin{defn}
A partition dangle is a disjoint union of the integers $\{1, \ldots, r\}$ into subsets $\chi_j$ (of cardinality $1\leq |\chi_j| \leq r$) each of which is annotated with an element of $\{ \ast, \circ\}$.   We refer to an annotated set $\chi_j$ as an arc, and to be more specific $\ast$-\emph{arcs} and $\circ$-\emph{arcs}.  An $(r,l)$-dangle is a dangle in which there are precisely $r-l$ of the $\ast$-arcs. 
\end{defn}
We can geometrically represent an $(r,l)$-dangle on a set of $r$ nodes by drawing an arc between the nodes labelled by the elements of the set $\chi_j$ and attaching a through-line to the $\ast$-arcs.

\begin{eg}
The algebra $P_K(2,\delta)$ has three layers in its iterated inflation structure.  The first layer is 2-dimensional with basis $V_0 \otimes V_0 \otimes \Sigma_2\simeq \Sigma_2$ with the obvious diagrammatic presentation. 

 The second layer is  $V_1 \otimes V_1 \otimes \Sigma_1$ where $V_1$ is 3-dimensional with basis $\{(1,2)^\ast\}, \{(1)^\ast, (2)^\circ\}$ and $ \{(1)^\circ, (2)^\ast\}$.  This layer is therefore 9-dimensional with diagrammatic presentation:
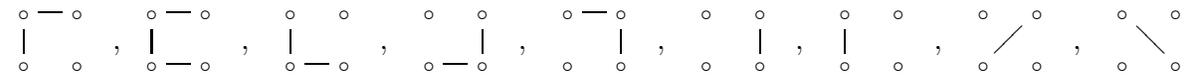
\begin{align*}
\begin{minipage}{54mm}
\def\objectstyle{\scriptstyle}
\xymatrix@=1pt{
 \circ \ar@{-}[rr] \ar@{-}[dd]	 &&  \circ   			\\
&&\\	
 \circ	 &&  \circ  }
\end{minipage} \  \ , \  
\begin{minipage}{54mm}
\def\objectstyle{\scriptstyle}
\xymatrix@=1pt{
 \circ \ar@{-}[rr] \ar@{-}[dd]	 &&  \circ   			\\
&&\\	
 \circ	 \ar@{-}[rr]&&  \circ  }
\end{minipage} \  \ , \  \
\begin{minipage}{54mm}
\def\objectstyle{\scriptstyle}
\xymatrix@=1pt{
 \circ   \ar@{-}[dd]	 &&  \circ   			\\
&&\\	
 \circ	 \ar@{-}[rr]&&  \circ  }
\end{minipage} \  \ , \  \
\begin{minipage}{54mm}
\def\objectstyle{\scriptstyle}
\xymatrix@=1pt{
 \circ   	 &&  \circ   	\ar@{-}[dd]		\\
&&\\	
 \circ	 \ar@{-}[rr]&&  \circ  }
\end{minipage} \  \ , \  \
\begin{minipage}{54mm}
\def\objectstyle{\scriptstyle}
\xymatrix@=1pt{
 \circ   	 \ar@{-}[rr] &&  \circ   	\ar@{-}[dd]		\\
&&\\	
 \circ	&&  \circ  }
\end{minipage} \  \ , \  \
\begin{minipage}{54mm}
\def\objectstyle{\scriptstyle}
\xymatrix@=1pt{
 \circ   	  &&  \circ   	\ar@{-}[dd]		\\
&&\\	
 \circ	&&  \circ  }
\end{minipage} \  \ , \  \
\begin{minipage}{54mm}
\def\objectstyle{\scriptstyle}
\xymatrix@=1pt{
 \circ   	\ar@{-}[dd]  &&  \circ   			\\
&&\\	
 \circ	&&  \circ  }
\end{minipage} \  \ , \  \
\begin{minipage}{54mm}
\def\objectstyle{\scriptstyle}
\xymatrix@=1pt{
 \circ   	  &&  \circ   	\ar@{-}[ddll]		\\
&&\\	
 \circ	&&  \circ  }
\end{minipage} \  \ , \  \
\begin{minipage}{54mm}
\def\objectstyle{\scriptstyle}
\xymatrix@=1pt{
 \circ   	\ar@{-}[ddrr]	  &&  \circ   		\\
&&\\	
 \circ	&&  \circ  }
\end{minipage}  
\end{align*}
The third layer is 4-dimensional with  $V_2 $ spanned by $\{(1)^\circ,(2)^\circ	\}$ and $\{(1, 2)^\circ	\}$.
\end{eg}
If $\delta \neq 0$ then we define $e_l$ to be $1/\delta$ times the basis element
\begin{align*}
\begin{minipage}{54mm}
\def\objectstyle{\scriptstyle}
\xymatrix@=1pt{
&& \circ 			&&\cdots	&&\circ  			&&\circ 		 	&&\circ  \ar@{-}[ll] \ar@{-}[rr]	 &&\cdots	&&\circ 	 	\ar@{-}[ll]&&\circ\ar@{-}[ll]  \\	&\\
&&\circ \ar@{-}[uu] 	&&\cdots	&&\circ  \ar@{-}[uu] 	&&\circ 		 	&&\circ  \ar@{-}[ll] \ar@{-}[rr]	 &&\cdots	&&\circ 	 	 \ar@{-}[ll]&&\circ \ar@{-}[ll]
}
\end{minipage}&, 
\end{align*}
with $r-l$ straight through-lines (coming from $\ast$-arcs of cardinality 1) and one northern (and one southern) $\circ$-arc joining the final $l$ northern (respectively southern) nodes. 
\subsection{The globalisation functor}\label{Brauerinflation}
We now discuss the globalisation and localisation functors for the partition algebra.  
The appearance of $\circ$-arcs of arbitrary length, $a$, is easy to cope with, if there exist $b$ such arcs, the stabiliser of such an element is $\Sigma_a \wr \Sigma_b$ and we inflate along the stabiliser as before.  

More interesting are the $\ast$-arcs.  The important point is that the left symmetric group action does not see the crossings of through-lines attached to arcs of \emph{different} lengths. 

\begin{eg} Consider for example, the $P_K(3,\delta)$-module $\G_1(M(1^2))$.  This module is 12-dimensional.  We consider the $\Sigma_3$-submodule spanned by elements which have a horizontal arc on the top row.  This module has the following basis: 
\begin{align*}
\begin{minipage}{54mm}
\def\objectstyle{\scriptstyle}
\xymatrix@=1pt{
 \circ\ar@{-}[dd]   	  &&  \circ   	\ar@{-}[dd]\ar@{-}[rr]	  &&  \circ 	\\
&&\\	
 \circ	&&  \circ   &&  \circ  }
\end{minipage} \ , \
\begin{minipage}{54mm}
\def\objectstyle{\scriptstyle}
\xymatrix@=1pt{
 \circ\ar@{-}@/_.75pc/[rrrr]   	  &&  \circ   	\ar@{-}[ddll]	  &&  \circ\ar@{-}[ddll] 	\\
&&\\	
 \circ	&&  \circ   &&  \circ  }
\end{minipage} \ , \
\begin{minipage}{54mm}
\def\objectstyle{\scriptstyle}
\xymatrix@=1pt{
 \circ 	  &&  \circ   	 \ar@{-}[ll]	\ar@{-}[dd]    &&  \circ \ar@{-}[ddllll]	\\
&&\\	
 \circ	&&  \circ   &&  \circ  }
\end{minipage} \ , \
\begin{minipage}{54mm}
\def\objectstyle{\scriptstyle}
\xymatrix@=1pt{
 \circ\ar@{-}[ddrr]   	  &&  \circ   	\ar@{-}[ddll]\ar@{-}[rr]	  &&  \circ 	\\
&&\\	
 \circ	&&  \circ   &&  \circ  }
\end{minipage} \ , \
\begin{minipage}{54mm}
\def\objectstyle{\scriptstyle}
\xymatrix@=1pt{
 \circ\ar@{-}[dd]\ar@{-}@/_.75pc/[rrrr]   	  &&  \circ   	\ar@{-}[dd]	  &&  \circ 	\\
&&\\	
 \circ	&&  \circ   &&  \circ  }
\end{minipage} \ , \
\begin{minipage}{54mm}
\def\objectstyle{\scriptstyle}
\xymatrix@=1pt{
 \circ\ar@{-}[dd]   	  &&  \circ   	 \ar@{-}[ll]	  &&  \circ \ar@{-}[ddll]	\\
&&\\	
 \circ	&&  \circ   &&  \circ  }
\end{minipage}.
\end{align*}

As a right $\Sigma_2$-module this decomposes as a direct sum of 3 copies of $M(1^2)$ (one coming from each of the 3 arc configurations on the top).
As a left $\Sigma_3$-module this decomposes as  direct sum of 2 copies of $M(2,1)$ (one for each of the two bases elements of $M(1^2)$).  It is indecomposable as a $\Sigma_3-\Sigma_2$-bimodule.
\end{eg}

We decompose the $\Sigma_s-\Sigma_t$ bi-module, ${^mV^l}$, as a direct sum of permutation modules for $\Sigma_s$ labelled by bi-compositions.   Assume that $m-l\geq0$ and let  $(\xi_\ast, \xi_\circ)$ be a bi-composition such that  $\xi_\ast \vDash l$ and $\sum_{i}i(\xi_\ast^i+\xi_\circ^i)=m$.  We shall define ${^mV^l_{(\xi_\ast,\xi_\circ)}}$ to be the cyclic $\Sigma_s -\Sigma_t$-bimodule generated by the diagram $v\otimes \epsilon_l \otimes b$ defined as follows.  Let $v$ be the dangle
which consists of $r$ nodes, with a wall  separating the first $\sum_i i\xi_\ast$ nodes from the next $\sum_i i\xi_\circ$ nodes;
the northern left (respectively right) hand side of the wall consists of $\xi_\ast^i$ (respectively $\xi_\circ^i$) $*$-arcs (respectively $\circ$-arcs) of length $i$ for each $i$, arranged from left to right in order of increasing length (this leaves $m$ remaining nodes to the far right, which are all connected by one $\circ$-arc).  
Let $b$ be the identity element of $\Sigma_{r-l}$ with no crossing through-lines, and let $\epsilon_l$ be the dangle with the first $r-l$ nodes free, and the final $l$ nodes connected to one another by one $\circ$-arc.  
 
 \begin{eg}
 Take $r=12$ and $m=0$, $l=8$.  Let $(\xi_\ast,\xi_\circ) = ((2,1,1),(0,1,1))$, then ${^{0}V^{8}_{(\xi_\ast,\xi_\circ)}}$ is the $\Sigma_{12}\times \Sigma_4$ bi-module generated by
 \begin{align*}
\begin{minipage}{54mm}
\def\objectstyle{\scriptstyle}
\xymatrix@=1pt{
  \circ\ar@{-}[dd]  	  & & \circ\ar@{-}[dd]  	  & &  \circ\ar@{-}[dd]\ar@{-}[rr]   		&&\circ  &&  \circ\ar@{-}[ddll] \ar@{-}[rr]&&  \circ\ar@{-}[rr]   & & \circ&&  \circ\ar@{-}[rr]& &  \circ&&\circ\ar@{-}[rr]&&\circ\ar@{-}[rr]&&\circ 	\\
&&&&& &&&&&&&&& &&&& && \\	
 \circ 	  & & \circ 	  & &  \circ  		&&\circ  &&  \circ  \ar@{-}[rr]   &&  \circ\ar@{-}[rr]   & & \circ\ar@{-}[rr]   &&  \circ\ar@{-}[rr]   & &  \circ\ar@{-}[rr]   &&\circ\ar@{-}[rr]&&\circ\ar@{-}[rr]   &&\circ    \\ }
\end{minipage}.
 \end{align*}
Note that we have omitted the wall separating the first 7 nodes from the final 5 nodes.
 \end{eg}

We let $\Sigma_\ast \wr \Sigma_{\xi_\ast}$ denote the group $\Sigma_1 \wr \Sigma_{\xi_\ast^1}\times \Sigma_2 \wr\Sigma_{\xi_\ast^2} \times \ldots $
and let $\Sigma_{\xi_\ast}$ denote the subgroup $\Sigma_{\xi_\ast^1}\times \Sigma_{\xi_\ast^2} \times \ldots $ (and extend this notation to $\xi_\circ$).

 \begin{prop}\label{fannyflaps}
As a $\Sigma_s-\Sigma_t$-bimodule ${^mV^l}=e_m(A/Ae_{l+1}A)e_l = \oplus_{(\xi_\ast,\xi_\circ)}{^mV^l_{(\xi_\ast,\xi_\circ)}}$.  This bi-modules give rise to a globalisation functor
\begin{align*}
{^mV^l_{(\xi_\ast,\xi_\circ)}}\otimes_{\Sigma_{t}} - : \Sigma_{t}{\text{\rm -mod}} \xrightarrow{\rm restriction} \Sigma_{\xi_\ast } \xrightarrow{\rm inflation}
  \Sigma_{\ast} \wr\Sigma _{\xi_\ast} \times \Sigma_{\circ} \wr\Sigma _{\xi_{\circ}}  {\text{\rm -mod}} 
 \xrightarrow{\rm induction}\Sigma_{s}{\text{\rm -mod}}  
 \end{align*}
 and a localisation functors
\begin{align*}
 {e_{(\xi_\ast,\xi_\circ)}}  : \Sigma_{s}{\text{\rm -mod}} \xrightarrow{\rm restriction}\Sigma_{\ast} \wr\Sigma _{\xi_\ast} \times \Sigma_{\circ} \wr\Sigma _{\xi_{\circ}}  {\text{\rm -mod}} \xrightarrow{\rm projection} \Sigma_{\xi_\ast } \xrightarrow{\rm induction} \Sigma_{t}{\text{\rm -mod}},
 \end{align*}
Furthermore, $J_{l,i}/J_{l,i+1} \otimes_{K\Sigma_{t}}M \simeq J_{l+i,0}/J_{l+i,1} \otimes_{K\Sigma_{t-i}}e_{l+i}M
 $
\end{prop}
\begin{proof}
 As a left $\Sigma_s$-module, ${^mV^l_{(\xi_\ast,\xi_\circ)}}$ is isomorphic to (multiple copies of) $K\!\uparrow_{\Sigma_\ast\wr\Sigma_{\xi_\ast}\times\Sigma_\circ\wr\Sigma_{\xi_\circ}}^{\Sigma_s}$, this is because 
 the left action of $\Sigma_s$  cannot swap through-lines attached to arcs of different length.
 As a right $\Sigma_t$-module ${^mV^l_{(\xi_\ast,\xi_\circ)}}$  is isomorphic to (multiple copies of)  $K\!\uparrow_{\Sigma_{\xi_\ast}}^{\Sigma_t}$ (for a similar reason).  We have that ${^mV^l_{(\xi_\ast,\xi_\circ)}}$ is indecomposable as a bimodule (therefore we do not obtain any multiplicities in the statement of the Proposition) and so the result follows.
 %
%
%
\end{proof}

\subsection{Modified (bi)-diagrams and (bi)-tableaux}
 In the case of the partition algebra we shall always take $r=t+l=s+m$, and $\lambda \vdash t$, $\mu \vdash s$. We let $i = m-l \in \mathbb{Z}$.

 \subsubsection{Modified diagrams}
We let a \emph{configuration of arcs}, $\Xi$, be any disjoint union of the integers $\{1, \ldots, r\}$ into subsets $\chi_j$ (of cardinality $1\leq |\chi_j| \leq t$) each of which is annotated with an element of $\{ \ast, \circ\}$.   We refer to an annotated set $\chi_j$ as an \emph{arc}.

For $\lambda \vdash t$, take the Young diagram $[\lambda]$; we record this configuration on $[\lambda]$ as follows: draw a line between the nodes labelled by the elements of the set $\chi_j$, (this line will not be unique); and then replace the nodes with the element of $\{ \ast, \circ\}$ that is attached to $\chi_j$.  For example, let $t=8$, take the configuration $\{1,7\}^\ast, \{2,8\}^\ast, \{3,4\}^\circ, \{5\}^\ast, \{6\}^\ast$, and record this on the Young diagram of $\lambda = (5,3)$ as follows:
\begin{align*}
\sigma=  \begin{minipage}{44mm}
\def\objectstyle{\scriptstyle}
\xymatrix@=1pt{
 \ast	 \ar@{-}[ddrr] 	&& \ast \ar@{-}[drdr] 		 && \circ \ar@{-}[rr]	&& \circ	&& \ast 	\\
 &\\
\ast	 	&&\ast 		&& \ast	 	
 }
\end{minipage} ,
\end{align*}

A \emph{modified diagram} arises from a pair  $(\lambda, i)$ as follows: take the Young diagram $[\lambda]$ and endow it with any configuration of annotated arcs such that $t-i$ equals the number of arcs $\chi_j$ which are annotated with a $\ast$.
For example,
\begin{align*}
 \begin{minipage}{44mm}
\def\objectstyle{\scriptstyle}
\xymatrix@=1pt{   \ast  		&&	\ast\ar@{-}[rr]   	  &&	\ast	  	\\
 &&\\ 
\ast 	&&	\ast \ar@{-}[uu] }
\end{minipage}  
\  \ , \  \ 
\begin{minipage}{44mm}
\def\objectstyle{\scriptstyle} 
\xymatrix@=1pt{ 
   \ast  	 	 &&	\circ	  &&	\ast	 	\\ &&\\ 
\ast  
	 &&	\circ \ar@{-}[uu] 	 }
\end{minipage} 
\  \ , \  \
\begin{minipage}{34mm}
\def\objectstyle{\scriptstyle}
\xymatrix@=1pt{
   \ast  		 &&		\ast \ar@{-}[ll]  	 &&	\ast	\\ &&\\ 
\ast 
	 && \circ  }
\end{minipage} 
\text{ are $((3,2),2)$-diagrams.}
\end{align*}
Define the equivalence relation,  row-standard modified diagrams, and cosets as before.   

 \subsubsection{Modified bi-diagrams}
For $\mu \vdash r-m$, define a $(\mu,m)$ bi-diagram to be the bi-diagrams $([\mu],[m])$ with a configuration of arcs (as above) under the additional constraints that (a) no $\circ$-arc has a node in $[\mu]$ (b) any $\ast$-arc is connected to exactly one node of $[\mu]$.  For example, there are three  $((2,1),1)$ bi-diagrams:
\begin{align*}
\omega_a=\begin{minipage}{44mm}
\def\objectstyle{\scriptstyle}
\xymatrix@=1pt{
  \ast	 	&& \ast 	  & 	&&&  \circ  \\
 &\\
 \ast	 	 	 		
 }
\end{minipage} \ , \
\omega_b=\begin{minipage}{44mm}
\def\objectstyle{\scriptstyle}
\xymatrix@=1pt{
 &\ast	 	&& \ast 	  & 	&&&  \ast \ar@{-}@/^.43pc/[llll] 	  	\\
 &\\
&\ast	 	 	 		
 }
\end{minipage} \ , \
\omega_c=\begin{minipage}{44mm}
\def\objectstyle{\scriptstyle}
\xymatrix@=1pt{
& \ast	 	&& \ast 	  & 	&&&  \ast \ar@{-}@/^.43pc/[ldlldlll] 	 	\\
 &\\
&\ast	 	 	 		
 }
\end{minipage}. 
\end{align*}
We let $T^{m}_{\mu,\rm{bi}}$ denote the set of all $(\mu,m)$ bi-diagrams.

\subsubsection{Restricted (bi)-diagrams and (bi)-tableaux} Let  $\sigma \in T^i_\lambda$, $\omega \in T^i_{\mu, \text{bi}}$, and let $\Xi$ be the configuration of arcs.  Fix  an integer $n$ and an element $\{\ast,\circ\}$.  Let $\Xi_n^\ast$ (or respectively $\Xi_n^\circ$) denote the set of all arcs, $\chi_j$, such that $|\chi_j|=n$ and that $\chi_j$ is annotated with a $\ast$ (or respectively a $\circ$).  For $\sigma \in T^i_\lambda$ (respectively $\omega \in T^i_{\mu, \text{bi}}$)  we let $|\sigma|$ (respectively $|\omega|$)  be the bi-composition $( (|\Xi^\ast_1|, (|\Xi^\ast_2|,\ldots ),(|\Xi^\circ_1|, (|\Xi^\circ_2|,\ldots ))$.

A $(\lambda,i)$-tableau (respectively $(\omega,i)$ bi-tableau) arises from a row standard $(\lambda,i)$-diagram (respectively $(\omega,i)$ bi-diagram) by replacing the annotations $\ast$ on elements of $\Xi_n^\ast$ with the labels $\{1 , \ldots, |\Xi^\ast_n| \}$.  In the above example, we have that 
$\{1,7\}^{{1} }$, $\{2,8\}^{{2} }$, $\{3,4\}^\circ,$ $\{5\}^{{2} },$ $\{6\}^{{1} }$,
  is a $((5,3),4)$-tableau.  We record this on the diagram as follows:
\begin{align*}
\sigma=  \begin{minipage}{44mm}
\def\objectstyle{\scriptstyle}
\xymatrix@=4pt{
 \ast	 \ar@{-}[ddrr] |\hole|{\textbf{1}}|\hole		&& \ast \ar@{-}[drdr]|\hole|{\textbf{2}}|\hole 		 && \circ \ar@{-}[rr]	&& \circ	&& {\textbf{2}}	\\
 &\\
{\textbf{1}} &&\ast 		&& \ast	 	
 }
\end{minipage} ,
 \end{align*}

For a given row standard $(\lambda,i)$ diagram, $\sigma\in T^i_\lambda$, (respectively bi-diagram $\omega \in T^{i}_{\mu,\rm{bi}}$) we shall now define the corresponding restricted, (respectively induced) diagram $\lambda\!\!\downarrow_\sigma$ (respectively bi-diagram $\mu\!\uparrow^\omega$).   
 For $n \geq1$, two arcs in $\Xi_n^\ast$ are conjugate if there exists an element of $\Sigma_\lambda$ (respectively $\Sigma_\mu \times \Sigma_m$) which swaps these two arcs.
  Take the cardinalities of a full set of conjugacy classes under this action, and arrange them in decreasing size, this gives us a partition which we denote $\lambda_\sigma^n$ (respectively  $\mu^\omega_{n,\ast}$).  Repeat with $\Xi^\circ_n$ to define $\mu^\omega_{n,\circ}$ similarly (note we do not do this for $\lambda$).
  
   We define the 
\emph{restricted diagram} $\lambda\!\!\downarrow_\sigma$ to be the multi-partition $(\lambda_\sigma^1, \lambda_\sigma^2, \ldots )$.
We let $\mu^\omega_{\ast}= (\mu^\omega_{1,\ast},\mu^\omega_{2,\ast},\ldots) $ and $\mu^\omega_{\circ}= (\mu^\omega_{1,\circ},\mu^\omega_{2,\circ},\ldots) $ and define the \emph{induced diagram} to be $\mu\!\uparrow^\omega=(\mu^\omega_{\ast};\mu^\omega_{\circ})$.
We let $\Sigma_{\lambda\downarrow_\sigma}$ denote the Young subgroup corresponding to the multi-partition $\lambda\!\!\downarrow_\sigma$.
We let  $\Sigma_\ast \wr\Sigma_{\mu^\omega_{\ast}}$ denote the group $\Sigma_1 \wr \Sigma_{\mu^\omega_{1,\ast}}\times \Sigma_2 \wr\Sigma_{\mu^\omega_{2,\ast}} \times \ldots $
and let $\Sigma_{\omega_\ast}$ denote the subgroup $\Sigma_{\mu^\omega_{1,\ast}}\times \Sigma_{\mu^\omega_{2,\ast}}\times \ldots $; we extend this notation to $\omega_\circ$ and let $\Sigma_{\mu\uparrow^\omega}$ denote the product $\Sigma_\ast \wr \Sigma_{\mu^\omega_\ast} \times  \Sigma_\circ\wr \Sigma_{\mu^\omega_\circ}$.  We let $M(\mu\!\uparrow^\omega)$ denote the permutation module on the cosets of  $\Sigma_{\mu\uparrow^\omega}$.

\begin{eg}
In the above examples $\lambda\!\!\downarrow_\sigma = ((1^2), (2))$,  $\mu\!\!\uparrow^{\omega_a} = (( (2,1) );(1))$,  $\mu\!\!\uparrow^{\omega_b} = (((1^2),(1));\o)$ and  $\mu\!\!\uparrow^{\omega_c} = (((2, 1));\o)$.   Therefore $\Sigma_{\lambda\downarrow_\sigma} =\Sigma_{2,1,1} = \Sigma_{\mu\uparrow^{\omega_a} }$, $\Sigma_{\mu\uparrow^{\omega_b} }=\Sigma_{1,1} \times \Sigma_2$, and $\Sigma_{\mu\uparrow^{\omega_c} }=\Sigma_{2} \times \Sigma_2$ (here we have used the fact that $\Sigma_2 \wr \Sigma_1 \simeq \Sigma_2$).
 \end{eg}
By  Proposition \ref{fannyflaps} 
we have that
 $	{^{m+i}V_{(\xi_\ast,\xi_\circ)}
^m}\otimes_{K\Sigma_{s}}M(\mu) \simeq \bigoplus_{|\omega| = (\xi_\ast,\xi_\circ)  }M(\mu\!\uparrow^\omega),$ 
summing over all possible arc configurations, we have that
$$	{^{m+i}V 
^m}\otimes_{K\Sigma_{s}}M(\mu) \simeq \oplus_{\omega\in T^{\rm{bi}}_{\mu,m}  }M(\mu\!\uparrow^\omega).$$  
 The dominance order, construction of semistandard tableaux etc. are all defined analogously to the Brauer algebra case.   
 
 \begin{lem}\label{propflap}

A set of double coset representatives for $\Sigma_{\lambda} \backslash \Sigma_{t} /\Sigma_{\mu\uparrow^\omega}$ is given by $\mathcal{D}_{\lambda\downarrow, \mu}^i = \{ \sigma \otimes \epsilon_{l+i} \otimes d_{\lambda\downarrow_\sigma, \mu} : \sigma \in T^i_\lambda \text{ such that }|\sigma|= |\omega| , d_{\lambda\downarrow_\sigma, \mu} \in \mathcal{D}_{\lambda\downarrow_\sigma, \mu}\}$.

A set of coset representatives for $\Sigma_{\lambda} /  \Sigma_{\lambda}\cap \Sigma_{\mu\uparrow^\omega}$ is given by $\mathcal{D}^i_\mu=\{\sigma \otimes \epsilon_{l+i} \otimes d_{\mu} : \sigma \in T^i_{1^t}   \text{ such that }|\sigma|= |\omega|, d_{\mu} \in \mathcal{D}_\mu\cap \Sigma_{\lambda \downarrow_\sigma} \}.$
\end{lem}

\begin{proof}

As in the previous sections, we have constructed the $(\lambda,i)$-diagrams as $\Sigma_\lambda$-orbits.  Note that in this section $\Sigma_t$ does not act transitively on the set of diagrams, $T^i_{1^t}$, but rather $\Sigma_t$ acts transitively on $\{\sigma\in T^i_{1^t}: |\sigma|=(\xi_\ast,\xi_\circ)\}$ for a fixed $(\xi_\ast,\xi_\circ)$.

For a fixed $(\xi_\ast,\xi_\circ)$, the $\Sigma_\lambda$ orbits on $\Sigma_t/\Sigma_\ast \wr \Sigma_{\xi_\ast}\times \Sigma_\circ  \wr \Sigma_{\xi_\circ}$ are parameterised by the diagrams $\{\sigma\in T^i_\lambda: |\sigma|=(\xi_\ast,\xi_\circ)\}$ by construction.   One can now proceed as before.
 \end{proof}

\subsection{Combinatorics of permutation modules}
We wish to consider the $\Sigma_s-\Sigma_t$-bimodule structure of ${^mV^l_{(\xi_\ast,\xi_\circ)}}\otimes_{K\Sigma_{r-l}}M(\lambda)$ and $e_{(\xi_\ast,\xi_\circ)}M(\lambda)$.  
\begin{prop}\label{tripdrip}
The permutation modules, $M(\lambda,l)$, for the partition algebra have a filtration with subquotients of the form $\oplus_{\sigma \in T^i_\lambda}\G_{l+i} (M(\lambda\!\downarrow_\sigma))$.  
 \end{prop}
\begin{proof}
By Proposition \ref{fannyflaps}, all that remains is to show that 
$$e_{l+i}M(\lambda) \simeq \oplus_{(\xi_\ast,\xi_\circ)}e_{(\xi_\ast,\xi_\circ)}M(\lambda)  \simeq \oplus_{\sigma\in T^i_\lambda} M(\lambda\!\downarrow_\sigma).$$   For each $(\xi_\ast,\xi_\circ)$ one may proceed as in the proof of Proposition \ref{flibbydejibjob} in order to show that $e_{(\xi_\ast,\xi_\circ)}M(\lambda) \simeq M(\lambda)/\langle m-hm:  h \in (\Sigma_\ast \wr \Sigma_{\xi_\ast} \times\Sigma_\circ \wr \Sigma_{\xi_\circ})\rangle$ is isomorphic to the sum $\oplus_{\sigma\in T^i_\lambda} M(\lambda\!\!\downarrow_\sigma)$ such that $|\sigma|  ={(\xi_\ast,\xi_\circ)}$.   The result now follows by summing over the  $e_{(\xi_\ast,\xi_\circ)}$.
 \end{proof}

We leave the proofs of the following corollaries as an exercise for the reader.

\begin{cor}[A Murphy basis]\label{partitionmurphy}
The permutation module $M(\lambda,l)$ has a basis in terms of modified tableaux as follows:
\begin{align*}
\{ v \otimes [\sigma] \otimes m_{ \mathfrak{s}\T^\sigma} :v \in {V_{l+i}}, \T^\sigma \in T^i_0(\omega, \lambda), \mathfrak{s} \in \Std(\omega) \text{ for some } \omega \vdash t\}.
\end{align*}
\end{cor}

\begin{cor}[Specht filtration]
The $P_K({r},\delta)$-module $M(\lambda,l)$ has a filtration $M(\lambda,l)=M_1 \supseteq M_2 \supseteq \ldots \supseteq M_{k+1} =0$ such that there exists $\mu^j\vdash t-i$ with $M_j/M_{j+1} \simeq \Delta(\mu^j,i)$. Moreover, for each partition $\mu$ the number of $\mu^j$ equal to $\mu$ is the number of semistandard $\mu$-tableaux of type $(\lambda,i)$.
\end{cor}
  \subsection{Homomorphisms between permutation modules}\label{partitionhom}
We have decomposed the $M(\mu,m)$ as a sum of transitive permutation modules for the symmetric group, and have constructed the cosets and double cosets required in order to prove the following.
 
  \begin{prop}\label{partitionhomthm}
Let $\lambda \vdash t$, $\mu \vdash s$ and $K$ be a field.  Then $\{\varphi_{\S^\sigma \T^\tau} : \S^\sigma \in T_0^\ast(\omega,\lambda), \T^\tau \in T_0^\ast(\omega,\mu)\}$ is a basis of $\Hom_{P_r}(M(\lambda,l), M(\mu,m))$, where $\varphi_{\S^\sigma\T^\tau}$ is given by
\begin{align*}
\varphi_{\S^\sigma \T^\tau} (\epsilon_l\otimes\epsilon_l\otimes x_\lambda) &= [\sigma]	\otimes[\tau]	\otimes m_{{\S  \T}},  
\end{align*}
\end{prop}
\begin{proof}
As in the Brauer case, we use Frobenius reciprocity to reduce to the question of determining symmetric group group homomorphisms between the various layers. In particular, by Proposition \ref{tripdrip} we must determine a basis for 
\begin{align*} 
\Hom_{K\Sigma_t}(M(\lambda), {^lV^{m+i}}\otimes_{K\Sigma_{r-m-i}}e_iM(\mu))  &\simeq 
 \oplus_{ \tau} \Hom_{K\Sigma_t}(M(\lambda),  {^lV^{m+i}}\otimes_{K\Sigma_{r-m-i}} M(\mu\!\downarrow_\tau) )\\
 &\simeq \oplus_{\tau \in T_\mu^i} \oplus_{ \omega \in T^{\rm{bi}}_\mu}  \Hom_{K\Sigma_t}(M(\lambda),   M((\mu\!\downarrow_\tau)\!\uparrow^\omega) )
\end{align*}
note that these homomorphisms are between two transitive permutation modules and so we can construct and index them via the (double) cosets given in Proposition \ref{propflap}, just as in the case of the proof of Proposition \ref{piranese}.  One can then substitute the Green--Dipper--James basis for the cellular basis, as in Proposition \ref{porg}.
 \end{proof}


\section{Semistandard bases of quasi-hereditary covers of diagram algebras}\label{homotheorem} 
We are now ready to show that the Schur algebras, $S(A)$, are cellular.  We shall show that these algebras are integral quasi-hereditary covers of  the Brauer, walled Brauer, and partition algebras and that they are 1-faithful in characteristic $p\neq 2,3$.

\begin{murthm}
Let $A$ be the classical Brauer, walled Brauer, or partition algebra.    Then the 
algebra $S(A)$ has a basis:
\begin{align*}
\Phi = \{ \varphi_{\S^\sigma\T^\tau} : \omega \vdash r-2n, \S^\sigma,\T^\tau \in {T}^\ast_0(\omega)			\},
\end{align*}
where $\S^\sigma, \T^\tau$ are $\omega$-tableaux of type $(\lambda,i)$ and $(\mu,j)$-tableaux respectively, and we define $\varphi_{\S^\sigma\T^\tau}$ to be the extension of the element of $\Hom_{A}(M(\lambda,l), M(\mu,m))$ given by
\begin{align*}
 \varphi_{\S^\sigma\T^\tau}(\epsilon_l \otimes \epsilon_l \otimes x_{\lambda}) = 
[\sigma] \otimes [\tau] \otimes m_{\S \T}.
\end{align*}
Moreover,
\begin{itemize}
\item The map $\ast: S(A) \to S(A)$, $:\varphi_{\S^\sigma\T^\tau} \to \varphi_{\T^\tau \S^\sigma}$ defines an anti-isomorphism of $S(A)$.
\item Suppose that $\omega \in \Lambda_A$ and that $\S$ is a semistandard $\omega$-tableau.  Then for all $\varphi 
 \in S(A)$ there exist $k_{\V^\nu} \in K$ such that for all $\T \in T_0^\ast(\omega)$
\begin{align*}
\varphi_{\S^\sigma\T^\tau}\circ \varphi 
=\sum_{\V \in T_0 (\lambda)}k_\V \varphi_{\S^\sigma\V^\nu} \quad \text{\rm{mod}-}\overline{S^\omega(A)},
\end{align*}
\end{itemize}
where the ideals $S^\omega(A), \overline{S^\omega(A)} \lhd S(A)$ are the $K$-modules spanned by $\varphi_{\S\T}$ such that $\S, \T \in T^\ast_0(\alpha)$ for some $\alpha \unrhd \omega$ or  $\alpha \rhd \omega$ respectively. 
 Consequently, this is a cellular basis of $S(A)$. 
\end{murthm}

\begin{proof}
We chose not to use the cellular basis in the statement of Proposition \ref{piranese}, but one can of course use this to determine the through-lines (as we did in Propositions and \ref{porg} \ref{partitionhomthm}) and hence $\Phi$ is a basis of the endomorphism algebra $S(A)$.
 
We want to study the composition of homomorphisms $\varphi_{\S^\sigma\T^\tau}\in \Hom_{B_r}(M(\mu,m))$ and $\varphi_{\U^\upsilon \V^\omega}\in \Hom_{B_r}(M(\mu,m), M(\nu,n))$, where $\S^\sigma \in T^\ast_0(\omega, \lambda)$, $\T^\tau \in T^\ast_0(\omega, \mu)$, $\U^\upsilon\in T^\ast_0(\omega', \mu)$, and $\V^\omega \in T^\ast_0(\omega', \nu)$.  We assume that $|\omega| \leq |\omega'|$ because the other case follows easily.
We have that $\varphi_{\S^\sigma\T^\tau}(\epsilon_l \otimes \epsilon_l \otimes x_{\lambda})=([\sigma] \otimes \tau \otimes h_{\S^\sigma})(\epsilon_m \otimes \epsilon_m\otimes x_{\mu})$ and $\varphi_{\U^\upsilon \V^\omega}(\epsilon_m \otimes \epsilon_m\otimes x_{\mu})=([\upsilon] \otimes \omega \otimes h_{\U^\upsilon})(\epsilon_n \otimes \epsilon_n\otimes x_{\nu})$ for some $h_{\S^\sigma}, h_{\U^\upsilon} \in \Sigma_r$.  

To prove that $\ast$ is an anti-isomorphism, we first check that
\begin{align*}\tag{$\star$}
 (\varphi_{\S^\sigma\T^\tau}(\epsilon_l \otimes \epsilon_l \otimes x_{\lambda}))^\ast &=( [\sigma] \otimes [\tau] \otimes m_{\S \T})^\ast 
\\ &= [\tau] \otimes [\sigma] \otimes m_{ \T \S}\\ &= \varphi_{\T^\tau \S^\sigma}(\epsilon_m \otimes \epsilon_m\otimes x_{\mu}).
\end{align*}
Therefore, we have that
\begin{align*}
(\varphi_{\S^\sigma\T^\tau} \circ \varphi_{U^{\upsilon} \V^\nu})^\ast(\epsilon_l \otimes \epsilon_l \otimes x_{\lambda})
 &=( ([\upsilon] \otimes \omega \otimes h_{\U^\upsilon})([\sigma] \otimes \tau \otimes h_{\S^\sigma})(\epsilon_l \otimes \epsilon_l \otimes x_{\lambda}) )^\ast \\
&=(\epsilon_l \otimes \epsilon_l \otimes x_{\lambda})( \tau \otimes[\sigma] \otimes h^\ast_{\S^\sigma})
(\omega \otimes [\upsilon]  \otimes h^\ast_{\U^\upsilon}),
\intertext{by $\star$ this is}
&=\varphi_{\T^\tau \S^\sigma}(\epsilon_l \otimes \epsilon_l \otimes x_{\lambda}) 		(\omega \otimes [\upsilon]  \otimes h^\ast_{\U^\upsilon}) \\
&=\varphi_{\T^\tau \S^\sigma}(	 ([\upsilon] \otimes \omega \otimes h_{\U^\upsilon})(\epsilon_l \otimes \epsilon_l \otimes x_{\mu})) ^\ast;
\intertext{again by $\star$ this is}
&=\varphi_{\T^\tau \S^\sigma}( \varphi_{ \V^\nu \U^{\upsilon}}	 (\epsilon_l \otimes \epsilon_l \otimes x_{\lambda})) 	\\
&=( \varphi_{ \V^\nu \U^{\upsilon}} \circ \varphi_{\T^\tau \S^\sigma})(\epsilon_l \otimes \epsilon_l \otimes x_{\lambda}),
\end{align*}
and so $\ast$ defines an anti-isomorphism.
Finally, we check the condition on multiplication.  We have that 
\begin{align*}
(\varphi_{\S^\sigma\T^\tau} \circ \varphi_{\U^{\upsilon} \V^\nu})(\epsilon_l \otimes \epsilon_l \otimes x_{\lambda})
&= ([\upsilon] \otimes \omega \otimes h_{\U^\upsilon})([\sigma] \otimes \tau \otimes h_{\S^\sigma})(\epsilon_l \otimes \epsilon_l \otimes x_{\lambda}) 
\\
&= ([\sigma] \otimes \tau \otimes h_{\S^\sigma})([\upsilon] \otimes [\omega] \otimes m_{\U^{\upsilon} \V^\nu}), 
\end{align*}
multiplication by elements of the Brauer algebra either increases the number of arcs, or acts by permutation of the through-lines and top rows.  Therefore, the composition is a sum of homomorphisms indexed by semistandard $\alpha$-tableaux, where $|\alpha| \leq |\omega|$.   By the definition of the dominance order we have that if $|\alpha| < |\omega|$, then we are done.  

If $|\alpha| = |\omega|$, then the multiplication by $ [\sigma] \otimes \tau \otimes h_{\S^\sigma}$  sends the top row of arcs of $([\upsilon] \otimes [\omega] \otimes m_{\U^{\upsilon} \V^\nu})$ to $[\sigma]$ and acts as an element $h\in \Sigma_r$ permuting the through-lines.
 This homomorphism is indexed by tableaux of type $(\lambda,\sigma)$ and $(\nu,\omega')$ such that $\lambda\!\!\downarrow_\sigma$, $\nu\!\! \downarrow_{\omega'} \vdash|\omega|$.  Finally, by Theorem \ref{homotheoremgreen} we have that $h(m_{\U^{\upsilon} \V^\nu})$ is expressible in terms of homomorphisms indexed by semistandard tableaux of type $\alpha \unrhd \omega$.  
\end{proof}

Recall that there exists a unique $\lambda$-tableau, $\T^\lambda$, of type $(\lambda,0)$.  Define the map $\varphi_\lambda= \varphi_{\T^\lambda\T^\lambda}+\overline{S^\lambda(A)}$.  This restricts to be the identity map on $M(\lambda,l)$.

\begin{defn}
Suppose that $(\lambda,l) \in \Lambda_A$, the Weyl module $W(\lambda,l)$ is the submodule of $S^\lambda(A)/\overline{S^\lambda(A)}$ generated by $\varphi_{ \T^\tau \T^\lambda } +\overline{S^\lambda(A)}$.
\end{defn}

If ${\S^\sigma}$ is a semistandard $(\lambda,i)$-tableau  let $\varphi_{\S^\sigma} = \varphi_{{\S^\sigma}\T
_\lambda} (\varphi_\lambda + \overline{S^\lambda(A)}) = \varphi_{{\S^\sigma}\T
_\lambda} + \overline{S^\lambda(A)}$. 

\begin{cor}
The Weyl module $W(\lambda,l)$ is a free $K$-module with basis $$\{\varphi_{\S^\sigma} : {\S^\sigma} \in T_0(\lambda,\mu)\text{ for some }(\mu,m) \in \Lambda_A\}.$$
\end{cor} 

Define a bilinear form $\langle \ \! \ \! , \ \rangle$ on $W(\lambda,l)$ by requiring that
$\varphi_{\T
_\lambda {\S^\sigma} } \varphi_{\T^\tau  \T
_\lambda} =\langle \varphi_{{\S^\sigma} },\varphi_{\T^\tau}\rangle \varphi_\lambda$ ${\rm mod}$-$\overline{S^\lambda(A)}$
for all semistandard $\lambda$-tableaux  ${\S}^\sigma$ and $\T^\tau$.

\begin{cor}
Let $A$ be the Brauer, walled Brauer, or partition algebra.   The algebra $S(A)$ is a quasi-hereditary cover of $A$. 

 The Schur functor induces a Morita equivalence in characteristic zero.
 In characteristic $p>3$ this cover is 1-faithful and therefore canonically associated to the diagram algebra.
\end{cor}

\begin{proof}
From the definition of the bilinear form on $W(\lambda,l)$, we have 
\begin{align*}
\varphi_{\T^\lambda\T^\lambda}\varphi_{\T^\lambda\T^\lambda} \equiv \langle \varphi_{\T^\lambda\T^\lambda} , \varphi_{\T^\lambda\T^\lambda}\rangle \varphi_\lambda \text{ mod } \overline{S^\lambda(A)}.
\end{align*}
However, $\varphi_{\T^\lambda\T^\lambda}\varphi_{\T^\lambda\T^\lambda}=\varphi_{\lambda}$ is the identity on $M(\lambda,l)$ and so $\langle \varphi_{\T^\lambda\T^\lambda} , \varphi_{\T^\lambda\T^\lambda}\rangle=1$.  Consequently each Weyl module $W(\lambda,l)$ is equipped with a non-vanishing bilinear form.  Therefore $D(\lambda,l) = \Head(W(\lambda,l))$ constitute a full set of non-isomorphic simple modules.  
  Therefore the cell-chain stratifies the algebra, and it is quasi-hereditary.

One half of the double centraliser property follows from the definition.  That $A=\End_{S(A)}(M)$ follows by consideration of the idempotent projection onto the permutation module $M(1^r)=A$.

We now show that the cover is 1-faithful for $p\neq2,3$.  By Theorem \ref{HHKPmain}, it is enough to check that $\Hom_A(M,-)$ is exact on modules with a cell-filtration when $p\neq 2,3$. 
The case of Brauer algebra is treated in \cite{Row}, we shall prove the result for the partition algebra and leave the case of the walled Brauer algebra as an exercise.
  
%

The partition algebra is a tower of recollement and so restriction of a cell module to a `Young subalgebra' has a cell-filtration, this generalises \cite[Proposition 7]{Row}. 

One can now repeat the arguments of \cite[Proposition 24]{Row} to give an inductive argument on the dominance ordering, this reduces the computation to showing that $\Ext^1_{P_K(r,\delta)}(K, \Delta(\mu))   =0. 
$ 

By Section \ref{Brauerinflation}, we have that $\Delta(\mu) = V_l \otimes \mathcal{S}(\mu)$ decomposes as a direct sum of $\Sigma_r$-modules, each of which is induced from a module for a product of various subgroups of the form $\Sigma_r\wr\Sigma_t$.  Each component module in the product is formed by inflation of a Specht module from $\Sigma_t$.  

We denote the $\Sigma_s \wr \Sigma_t$-module obtained by inflation from the $\Sigma_t$-module $M$, by ${\rm{infl}}_{\Sigma_t}^{\Sigma_s \wr {\Sigma_t}}M$.     
We only need show that $\Ext^1_{K\Sigma_s\wr\Sigma_{t}} (K,  {\rm infl}_{\Sigma_t}^{\Sigma_s \wr {\Sigma_t}}\mathcal{S}(\mu))=0$.    
The sequence
$$0 \to  (\Sigma_s)^{t} \to \Sigma_s \wr \Sigma_{t} \xrightarrow{\pi}\Sigma_{t} \to 0,$$
leads to a split quotient situation (as in \cite{reduction}) relating $\Sigma_s \wr \Sigma_{t}$ and $\Sigma_{t}$ and leads to the isomorphism $\Ext^1_{K\Sigma_s\wr\Sigma_{t}} (K,  {\rm infl}_{\Sigma_t}^{\Sigma_s \wr {\Sigma_t}}\mathcal{S}(\mu))=\Ext^1_{K  \Sigma_{t}} (K, \mathcal{S}(\mu))$, which is zero by the result of Hemmer and Nakano.  Here we have used that the trivial $\Sigma_s \wr \Sigma_{t}$-module is obtained by inflation of the trivial $\Sigma_{t}$-module.
 \end{proof}

\begin{rmk}  In the case that $A$ is the Brauer algebra, one should compare the results of this section with those of \cite[Sections 7-- 11]{hko}.   They prove directly that $S(A)$ is quasi-hereditary; whilst we  follow the Dipper--James--Mathas philosophy and first prove cellularity.
\end{rmk}

\begin{Acknowledgements*}
We thank Stuart Martin, Stephen Doty, and Rowena Paget for many helpful discussions on diagram algebras and their Schur--Weyl dualities.  We also thank David Stewart and the referee for their helpful comments on previous drafts.  
\end{Acknowledgements*}


\begin{thebibliography}\frenchspacing\raggedright\small
 

\bibitem[CPS96]{CPS} E.\ Cline; B.\ Parshall; L.\ Scott, Stratifying endomorphism algebras, Memoir A.~M.~S. \textbf{124} (1996).

\bibitem[CDM08]{coxwall} A. Cox; M. De Visscher; S. R, Doty; P. Martin, On the blocks of the walled Brauer algebra, J. Algebra \textbf{320} (2008), 169--212.
 
\bibitem[DJ86]{DJ} R. Dipper; G. D. James, Representations of Hecke algebras of general linear groups
Proc. London Math. Soc.   \textbf{52} (1986), no. 3, 20--52.

\bibitem[DJM98]{DJM}R.\ Dipper; G.\ James; A.\ Mathas, Cyclotomic q-Schur algebras, Math. Z. \textbf{229} (1998),
no. 3, 385--416. 

 \bibitem[DK07]{reduction}  L. Diracca; S. K\"onig, Cohomological reduction by split pairs, J. Pure Appl. Algebra  \textbf{212} (2008), no. 3, 471--485.

 
\bibitem[DT]{DT} S. Donkin; R. Tange, The Brauer algebra and the symplectic Schur algebra, Math. Z. \textbf{265} (2010), 187--219.

 

\bibitem[GL96]{GL} J. J. Graham; G. I. Lehrer, Cellular algebras, Invent. Math. \textbf{123} (1996), 1--34.

\bibitem[G80]{Green} J. A. Green,  {Polynomial Representations of  $\GL_n$}, vol. 830. Springer--Verlag, Berlin (1980).


\bibitem[HHKP10]{HHKP} R. Hartmann; A. Henke; S. K\"onig; R. Paget, Cohomological stratification of diagram algebras, Math. Ann. \textbf{347} (2010), 765--804.

 \bibitem[HP06]{Row} R. Hartmann; R. Paget, Young modules and filtration multiplicities for Brauer algebras,
Math. Z. \textbf{254} (2006), no. 2, 333--357.
 

\bibitem[HN04]{hemnak} D. J. Hemmer; D. K. Nakano, Specht filtration for Hecke algebras of type $A$, J. London Math. Soc. \textbf{69} (2004), no. 3, 623--638.  

\bibitem[HK11]{hko} A.\ Henke; S. K\"onig, Schur algebras of Brauer algebras I, DOI: 10.1007/s00209-011-0956-x


\bibitem[KX99]{KX} S. K\"{o}nig; C. Xi, Cellular algebras: inflations and Morita equivalences, {J.\ London Math.\ Soc.\ } \textbf{60} (1999), 700--722.

\bibitem[M92]{16} G. E. Murphy. On the representation theory of the symmetric groups and associated Hecke algebras, J. Algebra  \textbf{152} (1992), 492--513.
 
\bibitem[R95]{RAM} A. Ram, Characters of Brauer's centralizer algebras, Pacific J. Math. \textbf{169} (1995), no. 1, 173--200.

\bibitem[R08]{rouq} R. Rouquier, $q$-Schur algebras and complex reflection groups, Mosc. Math. J. \textbf{8} (2008),	 no. 1,  119--158.
 

\end{thebibliography}
\end{document}